\documentclass[12pt]{amsart}

\makeatletter
\def\thm@space@setup{
  \thm@preskip=6pt plus 2pt minus 1pt
  \thm@postskip=0pt
}
\makeatother

\makeatletter
\renewenvironment{proof}[1][\proofname]{
  \par
  \pushQED{\qed}
  \normalfont
  \topsep6\p@\@plus6\p@\relax
  \trivlist
  \item[\hskip\labelsep\itshape #1\@addpunct{.}]
  \ignorespaces
}{
  \par\nobreak\vspace{2pt}
  \popQED
  \medskip
  \endtrivlist
  \@endpefalse
}
\makeatother

\usepackage[margin=1in]{geometry}
\usepackage[T1]{fontenc}
\usepackage{lmodern}
\usepackage{microtype}

\usepackage{parskip}

\usepackage{amsmath,amssymb,mathtools}
\allowdisplaybreaks

\usepackage{enumitem}
\setlist[enumerate]{
  label=(\arabic*),
  leftmargin=*,
  itemsep=2pt
}
\setlist[itemize]{
  leftmargin=*,
  itemsep=2pt
}

\usepackage{array}
\usepackage{booktabs}
\usepackage{makecell}
\usepackage{multirow}
\usepackage{tabularx}
\usepackage{adjustbox}
\usepackage{graphicx}
\usepackage{float}
\usepackage{comment}

\usepackage[numbers,sort&compress]{natbib}

\usepackage{xcolor}
\usepackage[unicode]{hyperref}
\hypersetup{
  colorlinks=true,
  linkcolor=blue!50!black,
  citecolor=blue!50!black,
  urlcolor=blue!50!black
}
\usepackage[nameinlink,noabbrev]{cleveref}

\DeclareMathOperator{\aut}{Aut}

\numberwithin{equation}{section}
\numberwithin{figure}{section}
\numberwithin{table}{section}

\newtheorem{theorem}[equation]{Theorem}
\newtheorem{thm}[equation]{Theorem}
\newtheorem{prop}[equation]{Proposition}
\newtheorem{lem}[equation]{Lemma}
\newtheorem{cor}[equation]{Corollary}

\theoremstyle{definition}
\newtheorem{defn}[equation]{Definition}
\newtheorem{exm}[equation]{Example}

\theoremstyle{remark}
\newtheorem{rmk}[equation]{Remark}

\newcommand{\namedqedline}[1]{
  \par\vspace{0.5em}
  \noindent\hfill
  \fbox{\textbf{\scriptsize #1}}
  \par\vspace{0.25em}
}

\makeatletter
\let\c@figure\c@equation
\let\c@table\c@equation
\makeatother

\title[Terminal Geometry of Rank-Two Nielsen Reduction]
{Terminal Geometry of Rank-Two Nielsen Reduction on Simplicial Trees}

\author{Yukun Du}
\address[Yukun Du]{
Department of Mathematics, University of Georgia, Athens, GA 30603
}
\urladdr{https://www.math.uga.edu/directory/people/yukun-du}
\email{yukun.du@uga.edu}

\author{Sa'ar Hersonsky}
\address[Sa'ar Hersonsky]{
Department of Mathematics, University of Georgia, Athens, GA 30603
}
\urladdr{http://www.math.uga.edu/~saarh}
\email{saarh@uga.edu}



\date{\today}



\begin{document}

\begin{abstract}
Let $X$ be a locally finite simplicial tree, and let
$\Gamma=\langle\gamma_1,\gamma_2\rangle\leq\aut(X)$ be a discrete
subgroup generated by two hyperbolic automorphisms acting without
inversions, whose axes have finite edge intersection. We give an
explicit Euclidean localization of the first stopping stage in
rank-two Nielsen reduction and determine the exact
characteristic-set geometry of the resulting terminal Nielsen pair.

Let $m_1\geq m_2>0$ be the translation lengths of the normalized
initial generators, let $l$ denote their axis-overlap length, and
set $d=\gcd(m_1,m_2)$. If $l=0$, the initial pair is a Schottky
pair and $\Gamma$ is a discrete free group of rank two whose
nontrivial elements are all hyperbolic.
For $l>0$, the quotients and remainders of the Euclidean
algorithm for $m_1,m_2$ determine a finite partition of
$0<l\leq m_1+m_2-d$ that locates the first stopping state and
identifies its ordered Nielsen generators. Strict interior values
of these intervals yield terminal pairs satisfying the
short-overlap criterion. At nonfinal right endpoints, an additional
geometric invariant $l_0$ determines whether the terminal pair
is Schottky or contains a nontrivial elliptic generator. At the
final endpoint and for all larger overlaps, a nontrivial elliptic
terminal generator occurs.

We determine in every case the translation data and relative
positions of the axes or fixed subtrees of the terminal generators.
Every terminal pair is Nielsen-equivalent to the initial pair and
generates $\Gamma$. If both terminal generators are hyperbolic,
they form a Schottky pair and $\Gamma$ is a rank-two Schottky
group. Otherwise, exactly one terminal generator is nontrivial
elliptic. Local finiteness and discreteness imply that
this element has finite order, so $\Gamma$ contains nontrivial
torsion and is not free of rank two.
\end{abstract}

\maketitle

\setcounter{section}{-1}

\vspace{-1.2\baselineskip}

\vspace{-0.4\baselineskip}

\begingroup
\small
\setlength{\baselineskip}{0.8\baselineskip}
\setlength{\parskip}{0.375ex}
\tableofcontents
\endgroup

\section{Introduction}
\label{sec:introduction}

\subsection{Perspective}
\label{subsec:intro-perspective}

Let $X$ be a locally finite simplicial tree, equipped with its graph
metric. Throughout this paper, all automorphisms of $X$ under
consideration are assumed to be without inversions, meaning that
they do not interchange the endpoints of any edge. An automorphism
of $X$ is called elliptic if it fixes a vertex, equivalently if its
fixed-point set is a nonempty subtree, and hyperbolic if it preserves
a unique bi-infinite axis on which it acts by a nontrivial
translation. The characteristic set of an elliptic automorphism
is its fixed subtree, and the characteristic set of a hyperbolic
automorphism is its axis. For two hyperbolic automorphisms, their
translation lengths and the relative position of their axes provide
the most immediate geometric information about the subgroup they
generate.

The simplest configuration is governed by the standard ping-pong
argument. If the axes have no common edge, or if their common
segment is shorter than both translation lengths, then the two
automorphisms generate a free rank-two group acting freely and
properly discontinuously; see \cite[Lemma~2.6]{CuMo}.
In the locally finite simplicial setting used here, the
corresponding ping-pong formulation also gives discreteness in
the topology of pointwise convergence on the vertex set,
equivalently the permutation topology; see
\cite[Lemma~3.1]{Co1}. Thus, in this setting, the
subgroup is free and discrete, and every nontrivial element is
hyperbolic. We call such a subgroup a rank-two Schottky group,
and any of its ordered free bases a Schottky pair. The reduction
problem becomes more delicate when the axes have a finite
intersection whose length is at least the shorter translation
length.

A natural way to study this configuration is to apply Nielsen
transformations. In the strict reduction regime, the current
generators are hyperbolic, have unequal translation lengths,
and translate in the same direction along their common axis
segment. As in \cite[Proposition~4.1]{CuVo1},
if the overlap is strictly longer than the shorter translation
length, an appropriate Nielsen transformation decreases both
the longer translation length and the overlap by the shorter
translation length. Repeating this operation, and interchanging
the generators when necessary, produces a rank-two division
process.

One broad geometric framework for the analysis of this process
comes from the theory of actions on $\mathbb R$-trees. Culler and
Morgan developed the translation-length viewpoint and established
basic results concerning characteristic sets and their relation
to displacement and product lengths. In particular,
\cite[Proposition~1.3]{CuMo} describes characteristic
sets and displacement, Proposition~1.5 treats products associated
with disjoint characteristic sets or, when both automorphisms are
hyperbolic, with characteristic sets meeting at a single point,
and Proposition~1.8 detects nontrivial overlap and relative
orientation. Their Lemma~2.6 gives the short-overlap free-action
criterion. In the locally finite simplicial setting used here,
this criterion is the short-overlap case of the Schottky
conclusion stated below.

The seminal construction of Outer space by Culler and Vogtmann
\cite{CuVo} provides broader background through
marked metric graphs, free simplicial actions, and
translation-length functions. A direct predecessor of the
reduction considered in this paper appears in Culler--Vogtmann
\cite[Section~4]{CuVo1}. Starting with a basis of
$F_2$ whose elements act hyperbolically on an $\mathbb R$-tree,
they perform Nielsen changes of basis whenever the shorter
translation length does not exceed the overlap of the two axes.
Their Proposition~4.1 gives the corresponding subtraction of
overlap and, in the strict-overlap case, the subtraction of
translation length; it also includes the boundary case in which
the overlap equals the shorter translation length. The process
terminates when it produces either an elliptic basis element
or a basis whose overlap is shorter than both translation lengths.

Culler and Vogtmann further relate the sequence of basis changes
to the slope determined by the signed ratio of the two translation
lengths. In their analysis of circle-quotient actions, they
observe in the proof of Proposition~4.3 that the division
process, together with the corresponding sequence of Nielsen
transformations, is determined by this slope. In particular,
a rational slope forces termination. Proposition~4.3 itself
classifies the circle-quotient actions under consideration
by their slope, rather than giving a general classification
of Nielsen division processes; see
\cite[Proposition~4.3 and its proof]{CuVo1}.
In addition, in the remark following Proposition~4.2, they
observe that the difference between the overlap and the sum
of the two translation lengths remains invariant, and they
use this fact in
\cite[Propositions~4.4 and~4.5]{CuVo1}.
Their rank-two division process, its Euclidean character,
the strict subtraction rule, and the corresponding invariant
all motivate the question we consider in the paper.

A related algorithmic treatment was later developed by Conder. In
\cite{Co}, he gives a practical reduction algorithm
for two-generated subgroups acting on locally finite simplicial
trees. The procedure terminates by producing either an elliptic
element or a Nielsen-equivalent generating pair satisfying the
hypotheses of the ping-pong lemma, and therefore decides whether the
subgroup is discrete and free of rank two. The product-length
alternatives used in the algorithm, including an equality case
involving an additional translated-axis overlap, are given in
\cite[Proposition~3.5]{Co}.

Conder subsequently placed this question in a broader finite-tuple
minimality framework in \cite{Co1}. He conjectured that
every minimal tuple either contains an elliptic element or satisfies
the hypotheses of the ping-pong lemma, and proved the conjecture for
tuples of two and three generators; see Theorem~2.4 therein. For the
two-generator case, see in particular Lemma~7.1. Subject to the
general conjecture, this framework gives an algorithm for deciding
whether a finitely generated subgroup of the isometry group of a
locally finite simplicial tree is both discrete and free. 

More recently, Markowitz \cite{Ma} established
an unconditional recognition algorithm for arbitrary finite
generating sets of automorphisms of simplicial trees. Using a
different Nielsen-reduction procedure, the algorithm terminates
with either a free basis of a purely hyperbolic subgroup or a
generating set containing a nontrivial elliptic element. 
In particular, for locally finite trees, his algorithm decides
discreteness and freeness for arbitrary finite generating sets.
Our purpose here is different: we give an explicit Euclidean
localization of the first stopping state and determine the exact
terminal characteristic-set geometry of the rank-two Nielsen
division process.

Groups acting on locally finite trees also arise naturally in
non-Archimedean geometry, notably through Bruhat--Tits trees; see
\cite{Lu}. In this setting, free discrete subgroups
are closely related to Schottky uniformization and Mumford curves;
see \cite{GeVa}. From the viewpoint of intrinsic tree
geometry, Hersonsky and Hubbard \cite{HeHu}
studied finitely generated discrete subgroups of automorphism
groups of locally finite trees, including their limit sets and
critical exponents. The arguments in this paper, however, use
only the intrinsic geometry of locally finite simplicial trees
and their automorphisms.

\subsection{Main results and organization}
\label{subsec:intro-main-theorem}

The purpose of the present paper is to give an explicit Euclidean and
geometric resolution of the rank-two Nielsen division process in the
locally finite simplicial setting.

Let $\Gamma=\langle\gamma_1,\gamma_2\rangle\leq\aut(X)$ be a discrete
subgroup generated by two hyperbolic automorphisms whose axes have
finite intersection. We first normalize the initial pair. Set
$m_1=m_{\gamma_1}$ and $m_2=m_{\gamma_2}$ for their translation
lengths; after interchanging the generators if necessary, we let
$m_1\geq m_2>0$. Let $l$ be the length of the common edge segment
of the two axes, with $l=0$ when the axes have no common edge.
When $l>0$, we replace one generator by its inverse if necessary,
so that the generators translate in the same direction along this
segment. The ordered data $(l,m_1,m_2)$ will be called the geometric
triple of the normalized pair.

Since translation lengths in a simplicial tree are positive integers,
the formal length reductions are encoded by the ordinary Euclidean
algorithm. Its quotients and remainders determine a finite formal
sequence of Nielsen words in the normalized generators. At each
stage, strict reduction means the geometrically justified Nielsen
subtraction for a normalized hyperbolic pair with unequal translation
lengths $m_1>m_2>0$ and overlap $l>m_2$. This subtraction decreases
both the longer translation length and the overlap by $m_2$.

The geometric reduction realizes an initial segment of the formal
evolution, including the first state at which strict Nielsen
reduction no longer applies. We call this the first stopping state.
The formal numerical data agree with the actual geometric triples
at every realized state, through and including the first stopping
state. Any continuation of the Euclidean data beyond this state
is formal only; the data are no longer asserted to be the geometric
triples of the corresponding transformed generating pairs. In the
equality cases, the geometry after the prescribed additional
Nielsen transformation is determined separately by
Propositions~\ref{prop:reduce_iii} and~\ref{prop:reduce_iv}.

More precisely, for a normalized hyperbolic pair with
$m_1\geq m_2>0$, the stopping conditions are: short overlap
$0\leq l<m_2$, unequal-length equality $m_1>m_2$ and $l=m_2$,
or equal-length stopping $m_1=m_2$ and $l\geq m_2$.
In particular, zero overlap is included in the short-overlap
condition. Together with the strict reduction condition
$m_1>m_2$ and $l>m_2$, these cases exhaust the possible
geometric regimes.

Write $d=\gcd(m_1,m_2)$ for the greatest common divisor of the two
translation lengths. Equivalently, $d$ is the last nonzero remainder
in their Euclidean algorithm. Theorem~\ref{thm:termination} gives an
explicit Euclidean localization of the first stopping state. For
\begin{equation}
\label{eq:intro-euclidean-range}
0<l\leq m_1+m_2-d,
\end{equation}
the quotients and remainders in the Euclidean algorithm for
$m_1,m_2$ partition the possible values of the initial overlap
into finitely many intervals, whose right endpoints correspond
to equality stopping conditions. The interval containing $l$
selects a unique stage of the Euclidean evolution and thereby
identifies the ordered Nielsen pair at the first stopping state,
together with its translation lengths and remaining overlap.

We call the right endpoints of these intervals the Euclidean
endpoints, and the last one, equal to $m_1+m_2-d$, the final endpoint.
Interior points give the positive short-overlap stopping condition.
At a nonfinal Euclidean endpoint, the first stopping state satisfies
the unequal-length equality condition, while the final endpoint gives
the equal-length stopping condition. If $l>m_1+m_2-d$, the reduction
instead reaches the large-overlap equal-length stopping state
described in Theorem~\ref{thm:termination}.

We call the ordered pair at the first stopping state the stopping
pair. The terminal pair is the final ordered Nielsen pair specified
by Theorem~\ref{thm:terminal-geometry}. In the short-overlap case,
no further Nielsen subtraction is needed, although the stopping
pair may be interchanged to obtain the prescribed terminal ordering.
In either equality case, an additional Nielsen subtraction is
applied to the stopping pair, followed by the prescribed interchange.
Its geometric outcome is determined by the corresponding local
reduction proposition. Thus the stopping pair and the terminal pair refer to distinct
stages of the reduction process and need not coincide.

When $l>0$, let $v^+$ denote the endpoint of the initial common
segment toward which both normalized generators translate. During
strict reduction, the common orientation and the distinguished
endpoint $v^+$ are preserved. These facts allow us to match the
formal Nielsen words with the actual ordered pair throughout the
geometrically realized part of the process.

Theorem~\ref{thm:terminal-geometry} passes from the numerical first
stopping state to the terminal Nielsen pair and determines its exact
characteristic-set geometry. In the short-overlap case, the
terminal pair satisfies the Schottky criterion.

At a nonfinal Euclidean endpoint, write $(h,\eta)$ for the
unequal-length equality stopping pair and set
$l_0=\ell(A_h\cap\eta.A_h)$, with the convention $l_0=+\infty$
when the intersection contains an infinite ray. The prescribed
terminal Nielsen move produces the ordered pair
$(\eta,h\eta^{-1})$. The type of $h\eta^{-1}$ is determined by
the auxiliary-overlap test involving $l_0$ and the translation
lengths of the stopping pair. In the hyperbolic branch, we compute
its translation length and the distance between the terminal axes.
In the elliptic branch, we compute the distance from the remaining
hyperbolic axis to the fixed subtree of the elliptic terminal
generator.

At the final endpoint $l=m_1+m_2-d$, the terminal pair consists of
one hyperbolic and one elliptic generator, and the hyperbolic axis
meets the fixed subtree of the elliptic generator in exactly one
vertex. If $l>m_1+m_2-d$, the terminal pair again consists of one
hyperbolic and one elliptic generator, but the axis and fixed
subtree meet in a nondegenerate finite segment whose length is
computed explicitly.

In every elliptic case, the terminal elliptic generator is
nontrivial. Indeed, if it were the identity, Nielsen equivalence
would imply that $\Gamma$ is generated by the remaining hyperbolic
terminal generator and is therefore cyclic. The original
hyperbolic generators would then have the same axis, contradicting
the hypothesis that their axes have finite intersection.

The terminal alternatives are collected in
Table~\ref{tab:main-cases}. Since every terminal pair is
Nielsen-equivalent to the original pair, it generates $\Gamma$.
Our main theorem, Theorem~\ref{thm:main}, gives the final
group-theoretic classification. Either both terminal generators
are hyperbolic, in which case the terminal pair is Schottky and
$\Gamma$ is a rank-two Schottky group, or exactly one terminal
generator is nontrivial elliptic. In the latter case, discreteness
and local finiteness imply that this generator has finite order.
It follows that $\Gamma$ contains nontrivial torsion and therefore
is not free of rank two.

To make the reduction concrete, we introduce the running
Example~\ref{ex:running-example} in
Section~\ref{sec:schottky-criterion}. The example illustrates the
local reduction step and is revisited in
Section~\ref{sec:euclidean-resolution}, where the same evolution is
reconstructed from the Euclidean data and the corresponding Nielsen
words are identified. The remainder of the paper is organized as follows.

In Section~\ref{sec:preliminaries}, we fix the terminology 
and
notation for simplicial trees, automorphism types, characteristic
sets, translation lengths, discrete subgroups, Schottky pairs, and
Nielsen transformations. With this in place, Section~\ref{sec:schottky-criterion}
develops the local geometry of the reduction.
Proposition~\ref{prop:schottky} proves the short-overlap Schottky
criterion, while Proposition~\ref{prop:reduce_ii} treats strict
reduction and determines the exact new axis intersection, its
orientation, and its distinguished endpoint.
Proposition~\ref{prop:reduce_iii} then analyzes the unequal-length
equality case, and Proposition~\ref{prop:reduce_iv} treats the
equal-length terminal case. Having established the local reduction,
Section~\ref{sec:euclidean-resolution} combines these geometric
results with the Euclidean algorithm. Lemma~\ref{lem:euc} separates
the complete formal evolution from its geometrically realized initial
segment. Theorem~\ref{thm:termination} locates the first stopping state,
Theorem~\ref{thm:terminal-geometry} determines the terminal Nielsen
pair and its exact characteristic-set geometry, and our main theorem,
Theorem~\ref{thm:main}, gives the resulting Schottky versus torsion
classification.

\section{Preliminaries}
\label{sec:preliminaries}

This section fixes the terminology and basic geometric results used
throughout the paper. We first recall the classification of tree
automorphisms, the geometry of fixed subtrees and axes, and the
displacement formulas used to detect elliptic and hyperbolic behavior.
We then define discrete subgroups and Schottky pairs. The final part of
the section introduces the elementary Nielsen moves used in the
rank-two reduction process.

\subsection{Trees and tree automorphisms}
\label{subsec:trees}

A \emph{tree} is a connected, acyclic simple graph $X$ with vertex set
$\mathcal V(X)$ and edge set $\mathcal E(X)$. For vertices
$v,w\in\mathcal V(X)$, let $[v,w]$ denote the unique simple path
joining them. The number of edges in $[v,w]$ defines the combinatorial
distance
\begin{equation}
\label{eq:tree_combinatorial_distance}
d(v,w)=\ell([v,w]).
\end{equation}
We also use the corresponding open and half-open path notation;
for example, $[v,w)=[v,w]\setminus\{w\}$.

An automorphism of $X$ is a graph isomorphism $g\colon X\to X$. The
group of all automorphisms of $X$ is denoted by $\aut(X)$.

\begin{defn}[Types of tree automorphisms]
\label{def:aut-types}
Let $g\in\aut(X)$.
\begin{itemize}
\item The element $g$ is \emph{elliptic} if it fixes a vertex.
\item The element $g$ is an \emph{inversion} if it preserves an edge
and interchanges its endpoints.
\item The element $g$ is \emph{hyperbolic} if it is neither elliptic
nor an inversion.
\end{itemize}
\end{defn}

Throughout this paper, all automorphisms of $X$ under consideration
are assumed to be without inversions.

We will use the following classical facts from the geometry of trees;
see \cite[Ch.~I]{Se} and
\cite[Proposition~1.3]{CuMo}. For a nonempty subtree
$Y\subset X$ and $v\in\mathcal V(X)$, write
$d(v,Y)=\min\{d(v,w):w\in\mathcal V(Y)\}$.

\begin{prop}[Elliptic automorphisms]
\label{prop:elliptic}
Let $g\in\aut(X)$ be elliptic. Its fixed-point set is a nonempty
subtree (possibly a single vertex), denoted by $T_g$. For every
$v\in\mathcal V(X)$,
\begin{equation}
\label{eq:elliptic_displacement_formula}
d(v,g.v)=2d(v,T_g).
\end{equation}
\end{prop}

\begin{prop}[Hyperbolic automorphisms]
\label{prop:hyperbolic}
Let $g\in\aut(X)$ be hyperbolic. Its minimum displacement
\begin{equation}
\label{eq:hyperbolic_translation_length}
m_g
=
\min_{v\in\mathcal V(X)}d(v,g.v)
\in\mathbb N_{>0}
\end{equation}
is attained precisely on a unique bi-infinite geodesic $A_g$, called
the \emph{axis} of $g$. The element $g$ acts on $A_g$ as a translation
of length $m_g$. Moreover, for every $v\in\mathcal V(X)$,
\begin{equation}
\label{eq:hyperbolic_displacement_formula}
d(v,g.v)=m_g+2d(v,A_g).
\end{equation}
\end{prop}

The subtree $T_g$ in the elliptic case and the axis $A_g$ in the
hyperbolic case will be referred to as the characteristic set of $g$.

The preceding displacement formulas give the following useful
three-point criterion for determining the type of an automorphism.

\begin{cor}[Three-point criterion]
\label{cor:aut_type}
Let $g\in\aut(X)$ and $v\in\mathcal V(X)$. Then:
\begin{itemize}
\item $g$ is elliptic if and only if
\begin{equation}
\label{eq:automorphism_type_nonstrict}
d(g^{-1}.v,g.v)\leq d(v,g.v);
\end{equation}

\item $g$ is hyperbolic if and only if
\begin{equation}
\label{eq:automorphism_type_strict}
d(g^{-1}.v,g.v)>d(v,g.v).
\end{equation}
\end{itemize}
\end{cor}

\begin{proof}
Suppose first that $g$ is hyperbolic. Since $g^2$ has the
same axis $A_g$ and translation length $2m_g$, the displacement
formula applied to $g^2$ gives
\begin{equation}
d(g^{-1}.v,g.v)
=
d(v,g^2.v)
=
2m_g+2d(v,A_g).
\end{equation}
By (\ref{eq:hyperbolic_displacement_formula}),
\begin{equation}
d(v,g.v)=m_g+2d(v,A_g).
\end{equation}
Since $m_g>0$, it follows that
\begin{equation}
d(g^{-1}.v,g.v)>d(v,g.v).
\end{equation}

Suppose next that $g$ is elliptic. 
By (\ref{eq:elliptic_displacement_formula}),
\begin{equation}
d(v,g.v)=2d(v,T_g),
\end{equation}
and
\begin{equation}
d(g^{-1}.v,g.v)
=
d(v,g^2.v)
=
2d(v,T_{g^2}).
\end{equation}
Since $g^2$ is elliptic and
$T_g\subseteq T_{g^2}$,
\begin{equation}
d(v,T_{g^2})\leq d(v,T_g).
\end{equation}
Therefore,
\begin{equation}
d(g^{-1}.v,g.v)\leq d(v,g.v).
\end{equation}

By the standing assumption that all automorphisms of $X$ under
consideration are without inversions, every automorphism under
consideration is either elliptic or hyperbolic. The two inequalities
above therefore give both converses.
\end{proof}

We will also use the following consequences in order to locate the
fixed subtree or the axis from the orbit of a single vertex.

\begin{cor}[Locating the characteristic set]
\label{cor:aut_impl}
Let $g\in\aut(X)$ and $v\in\mathcal V(X)$.

If $g$ is elliptic and $u$ is the midpoint of $[v,g.v]$, then
\begin{equation}
\label{eq:elliptic_midpoint_intersection}
T_g\cap[v,g.v]=\{u\}.
\end{equation}

If $g$ is hyperbolic, let $w\in[v,g.v]$ be the point satisfying
\begin{equation}
\label{eq:hyperbolic_projection_position}
d(v,w)
=
d(v,g.v)
-
\frac{1}{2}d(g^{-1}.v,g.v).
\end{equation}
Then,
\begin{equation}
\label{eq:hyperbolic_axis_orbit_intersection}
A_g\cap[v,g.v]=[w,g.w].
\end{equation}

Moreover,
\begin{equation}
\label{eq:hyperbolic_length_from_three_points}
m_g
=
d(g^{-1}.v,g.v)-d(v,g.v).
\end{equation}
\end{cor}
\begin{proof}
Suppose first that $g$ is elliptic, and let $p$ be the projection of
$v$ to the fixed subtree $T_g$. Since $g$ fixes $T_g$ pointwise, we
have $g.p=p$. Moreover, $p$ is also the projection of $g.v$ to
$T_g$, because $g$ preserves distances and fixes $T_g$ pointwise.
Hence the geodesic $[v,g.v]$ passes through $p$. Since
$d(v,p)=d(g.v,p)$, the point $p$ is the midpoint $u$ of
$[v,g.v]$.

By (\ref{eq:elliptic_displacement_formula}),
$d(v,g.v)=2d(v,T_g)$, so $d(v,g.v)$ is even. Thus the midpoint $u$
is a vertex of $X$, rather than a point in the interior of an edge.
Since $p$ is the projection of $v$ to $T_g$, no other point of
$T_g$ can lie on the segment $[v,p)$. Applying $g$ gives the same
conclusion for $(p,g.v]$. Therefore
$T_g\cap[v,g.v]=\{u\}$, proving
(\ref{eq:elliptic_midpoint_intersection}).

Now suppose that $g$ is hyperbolic, and let $p$ be the projection of
$v$ to $A_g$. Set $a=d(v,A_g)=d(v,p)$. The geodesic $[v,g.v]$
passes from $v$ to $p$, then along $A_g$ from $p$ to $g.p$, and
finally from $g.p$ to $g.v$. Hence
$d(v,g.v)=m_g+2a$.

Likewise, the projections of $g^{-1}.v$ and $g.v$ to $A_g$ are
$g^{-1}.p$ and $g.p$, respectively. Since $g$ translates $A_g$ by
$m_g$, the segment of $A_g$ from $g^{-1}.p$ to $g.p$ has length
$2m_g$. Therefore
$d(g^{-1}.v,g.v)=2m_g+2a$.

Subtracting the first equality from the second gives
$m_g=d(g^{-1}.v,g.v)-d(v,g.v)$, which proves
(\ref{eq:hyperbolic_length_from_three_points}). Moreover,
\[
d(v,g.v)-\frac{1}{2}d(g^{-1}.v,g.v)=a=d(v,p).
\]
Since $p\in[v,g.v]$ and $p$ is the unique point of $[v,g.v]$ at
distance $a$ from $v$, the point $w$ in
(\ref{eq:hyperbolic_projection_position}) is precisely $p$. 
Since $p$ and $g.p$
are the projections of $v$ and $g.v$ to $A_g$, respectively,
\begin{equation}
A_g\cap[v,g.v]=[p,g.p]=[w,g.w].
\end{equation}

\end{proof}

We now record the notion of discreteness that will be used throughout
the paper.

\begin{defn}[Discrete subgroups]
\label{def:discrete}
A subgroup $G<\aut(X)$ is \emph{discrete} if there exists a nonempty
finite subtree $T\subset X$ whose pointwise stabilizer in $G$ is
trivial. That is,
\begin{equation}
\label{eq:discrete_pointwise_stabilizer}
G_T
=
\{g\in G\mid g.v=v
\text{ for every }v\in\mathcal V(T)\}
=
\{1\}.
\end{equation}
\end{defn}

\begin{rmk}[Finite vertex stabilizers]\label{rmk:finite-vertex-stabilizers}
Definition~\ref{def:discrete} agrees with discreteness in the
permutation topology on $\aut(X)$. Indeed, the pointwise stabilizers
of finite subsets of $\mathcal V(X)$ form a neighborhood basis of
the identity, and the convex hull of a finite set of vertices is a
finite subtree.

When $X$ is locally finite, every vertex stabilizer in $\aut(X)$ is
compact and open. Hence, if $G<\aut(X)$ is discrete and $v\in\mathcal V(X)$, 
then
\begin{equation}
\label{eq:finite_vertex_stabilizer}
G_v
=
G\cap\aut(X)_v
=
\{g\in G\mid g.v=v\}
\end{equation}
is a discrete subgroup of the compact group $\aut(X)_v$, and is
therefore finite.
\end{rmk}

Finally, we introduce the class of groups and generating pairs that
will appear in the main classification theorem.

\begin{defn}[Schottky groups and Schottky pairs]
\label{def:schottky-pair}
A subgroup $G<\aut(X)$ is called a \emph{Schottky group} if it is
finitely generated, free, and discrete, and if every nontrivial
element of $G$ is hyperbolic.

An ordered pair $(g_1,g_2)$ is called a \emph{Schottky pair} if
$g_1$ and $g_2$ form a free basis of a rank-two Schottky group. In
this case, $g_1$ and $g_2$ are called \emph{Schottky generators}.
\end{defn}

In particular, hyperbolicity of the two generators alone does not
imply that the pair is Schottky. In the short-overlap configurations
considered later, explicit ping-pong domains establish freeness and
discreteness. Together with local finiteness and the standing
no-inversions assumption, these properties also exclude nontrivial
elliptic elements and establish the Schottky-pair conclusion.

\subsection{Nielsen moves and length reduction}
\label{subsec:nielsen}

Let
$\Gamma=\langle\gamma_1,\gamma_2\rangle<\aut(X)$, where the initial
generators are hyperbolic. Let
$F_2=\langle x_1,x_2\rangle$ be the free group on two generators,
and let
$\phi:F_2\to\Gamma$ be the homomorphism defined by
$\phi(x_1)=\gamma_1$ and $\phi(x_2)=\gamma_2$. An elementary 
\emph{Nielsen move} on the
ordered pair $(\gamma_1,\gamma_2)$ is induced by an elementary
Nielsen automorphism of $F_2$ and is one of the transformations
\begin{equation}
\label{equ:nielsen-moves}
\begin{split}
(\gamma_1,\gamma_2)
&\longmapsto
(\gamma_1^{-1},\gamma_2),\\
(\gamma_1,\gamma_2)
&\longmapsto
(\gamma_2,\gamma_1),\\
(\gamma_1,\gamma_2)
&\longmapsto
(\gamma_1\gamma_2,\gamma_2),
\end{split}
\end{equation}
or the inverse of one of these transformations. These are induced by
a standard generating set of elementary Nielsen automorphisms of
$F_2$. Evaluating the transformed free generators at
$(\gamma_1,\gamma_2)$ gives the corresponding transformed pair in
$\Gamma$. In particular, Nielsen moves preserve the subgroup
generated by the ordered pair, whether or not $\Gamma$ is free.

The inversion and interchange moves are used to normalize the pair. 
After interchanging the generators if necessary, we arrange that 
$m_{\gamma_1}\geq m_{\gamma_2}$. If the two axes have a 
nondegenerate common segment, then after inverting a generator 
if necessary, we further arrange that the generators translate in 
the same direction along that segment. We call an ordered pair in 
this position normalized. Once the pair is normalized, the reduction 
uses a product move of the form
\begin{equation}
\label{eq:nielsen_subtraction_preliminary}
(\gamma_1,\gamma_2)
\longmapsto
(\gamma_1\gamma_2^{-1},\gamma_2).
\end{equation}

For an ordered pair of hyperbolic elements, define its total
translation length by
\begin{equation}
\label{eq:l_order}
\mathcal L(\gamma_1,\gamma_2)
\coloneqq
m_{\gamma_1}+m_{\gamma_2}.
\end{equation}

If a Nielsen move sends
$(\gamma_1,\gamma_2)$ to $(\gamma_1',\gamma_2')$, then it is called a
\emph{length-reducing Nielsen move} if the transformed pair consists
of hyperbolic elements and satisfies
\begin{equation}
\label{eq:n_l_reducing}
\mathcal L(\gamma_1',\gamma_2')
<
\mathcal L(\gamma_1,\gamma_2).
\end{equation}
We reserve the term \emph{strict reduction} for the specific
geometric subtraction step introduced below.

In order to describe the reduction used later, set
$m_1=m_{\gamma_1}$ and $m_2=m_{\gamma_2}$. Assume that
$m_1\geq m_2$, that the axes have a finite common segment of length
$l$, and that the two generators translate in the same direction
along that segment. Under the strict inequalities
\begin{equation}
\label{eq:preliminary_strict_reduction_conditions}
m_1>m_2,
\qquad
l>m_2,
\end{equation}
Proposition~\ref{prop:reduce_ii} proves that the Nielsen 
move in
(\ref{eq:nielsen_subtraction_preliminary}) produces a hyperbolic
element satisfying
\begin{equation}
\label{eq:preliminary_length_subtraction}
m_{\gamma_1\gamma_2^{-1}}=m_1-m_2.
\end{equation}
Consequently,
\begin{equation}
\label{eq:preliminary_total_length_drop}
\mathcal L(\gamma_1\gamma_2^{-1},\gamma_2)
=m_1<m_1+m_2=\mathcal L(\gamma_1,\gamma_2).
\end{equation}
Proposition~\ref{prop:reduce_ii} also shows that the 
axes of $\gamma_1\gamma_2^{-1}$ and $\gamma_2$ have
a common segment of length $l-m_2$. We call this Nielsen 
subtraction a \emph{strict reduction}, and we call $m_1>m_2$ and $l>m_2$ the 
strict reduction conditions.

Iterating the subtraction move and interchanging the two generators
when necessary produces a sequence of ordered pairs
\begin{equation}
\label{eq:preliminary_nielsen_sequence}
(\gamma_1,\gamma_2)
\longrightarrow
(\gamma_1^{(1)},\gamma_2^{(1)})
\longrightarrow
\cdots.
\end{equation}
As long as both elements remain hyperbolic and the strict reduction
conditions hold, the total translation length decreases. Since the
translation lengths of hyperbolic automorphisms of a simplicial tree
are positive integers, any sequence of strict reductions is finite.

After consecutive subtraction moves are grouped together, the
translation lengths follow the ordinary Euclidean algorithm up to
the geometric stopping stage. The first stopping state may satisfy
the short-overlap condition, the unequal-length equality condition,
or the equal-length condition. In the short-overlap case, the
stopping pair is already terminal. In either equality case, one
additional Nielsen move produces the terminal pair. At an
unequal-length equality value, an additional translated-axis overlap
is needed to determine whether the new terminal generator is
hyperbolic or elliptic. The local geometric analysis is carried out
in Section~\ref{sec:schottky-criterion}, and its explicit Euclidean
resolution is given in Section~\ref{sec:euclidean-resolution}.

\begin{rmk}[Relation with the Nielsen division process]
The elementary Nielsen division process used here is standard. For
the rank-two division process, see
\cite[Section~4]{CuVo1}. For the corresponding
algorithmic reduction of pairs of tree automorphisms, see
\cite{Co}, and for the later finite-tuple minimality
framework, see \cite{Co1}. Our later results use the
Euclidean quotients and remainders to identify the exact stopping
stage, the terminal generators, and their characteristic-set
geometry.
\end{rmk}

\section{Geometric Triples and Nielsen Reduction}\label{sec:schottky-criterion}

In this section, we consider pairs of hyperbolic tree automorphisms.
Throughout, we assume that
$\Gamma=\langle\gamma_1,\gamma_2\rangle$ acts on $X$ without inversions
and that the intersection of the two axes contains only finitely many
edges.

Recall from Section~\ref{subsec:nielsen} that whenever the strict
reduction conditions
(\ref{eq:preliminary_strict_reduction_conditions}) hold, a Nielsen
subtraction move decreases the total translation length. After
grouping consecutive subtraction moves, the translation lengths
follow the ordinary Euclidean algorithm until the first ordered pair
for which the strict reduction conditions fail.

We will classify this first failure of strict reduction and determine
the resulting terminal configurations of the characteristic sets.
First, we introduce the normalization used throughout this section.

\begin{itemize}
\item After interchanging the generators if necessary, assume that
$m_{\gamma_1}\geq m_{\gamma_2}$, where $m_{\gamma_i}$ is the
translation length defined in
(\ref{eq:hyperbolic_translation_length}).

\item If the axes have a common edge, then their intersection is a
finite segment. We write
\begin{equation}
\label{eq:axis_intersection_segment}
A_{\gamma_1}\cap A_{\gamma_2}=[v^-,v^+].
\end{equation}
After replacing one generator by its inverse if necessary, we assume
that both generators translate from $v^-$ toward $v^+$ along this
segment.

\item If the axes have no common edge, let $[v_1,v_2]$ be the unique
possibly degenerate geodesic joining them, where
$v_i\in A_{\gamma_i}$. In this case we say that the axes are
\emph{edge-disjoint}. Thus, edge-disjoint axes are allowed to meet at
one vertex.
\end{itemize}

\subsection{Geometric Triples and the Short-Overlap Criterion}
\label{subsec:geometric-triples}

We first introduce terminology for the geometric data associated to an
ordered generating pair.

\begin{defn}[Geometric triples]
\label{def:geometric-triple}
Let $X$ be a tree, and let $\gamma_1,\gamma_2\in\aut(X)$ be
hyperbolic automorphisms whose axis intersection contains only
finitely many edges. We write
\begin{equation}
\label{eq:initial_geometric_data}
m_i=m_{\gamma_i},
\qquad
l=\ell(A_{\gamma_1}\cap A_{\gamma_2}),
\end{equation}
where $l=0$ if the axes have no common edge. The triple
$(l,m_1,m_2)$ is called the \emph{geometric triple} of the ordered
generating pair $(\gamma_1,\gamma_2)$.
\end{defn}

The term ``geometric triple'' records only these numerical data and
does not determine the group-theoretic properties in general.

For a normalized pair, the geometric triple satisfies exactly one of
the following conditions:
\begin{enumerate}[label=(\Roman*)]
\item $l\geq m_2$;
\item $0<l<m_2$;
\item $l=0$.
\end{enumerate}

In cases (II) and (III), the short-overlap condition
$l<m_2\leq m_1$, together with local finiteness of $X$, implies
that the given pair is Schottky, as the following proposition shows.
For comparison, see \cite[Proposition 3.4]{Co}.
The proof establishes the full conclusion needed later, including
discreteness, freeness, the free-basis statement, and hyperbolicity
of every nontrivial element.

\begin{prop}[Short-overlap Schottky criterion, 
    {\cite[Lemma~2.6]{CuMo}}]
\label{prop:schottky}
Let $X$ be a locally finite tree, and let
$\gamma_1,\gamma_2\in\aut(X)$ be hyperbolic automorphisms with
translation lengths $m_1\geq m_2$. Suppose that
\begin{equation}
\label{eq:short_overlap_condition}
\ell(A_{\gamma_1}\cap A_{\gamma_2})<m_2.
\end{equation}
Equivalently, the associated data fall into case (II) or case (III).
Then $(\gamma_1,\gamma_2)$ is a Schottky pair. In particular,
$\Gamma$ is discrete and free of rank two.
\end{prop}

\subsection{Formal Reduction Data and Strict Nielsen Reduction}
\label{se:form_red_str_nie_red}

Suppose that the geometric triple $(l,m_1,m_2)$ is in case (I).
We apply the Nielsen transformation
\begin{equation}
\label{eq:nielsen_subtraction_pair}
(\gamma_1,\gamma_2)
\longmapsto
(\gamma_1\gamma_2^{-1},\gamma_2).
\end{equation}
The resulting pair generates the same subgroup $\Gamma$. Motivated by
the expected change in translation lengths and axis overlap, define the
formal reduction data
\begin{equation}
\label{eq:triple}
(l',m_1',m_2)
=
(l-m_2,m_1-m_2,m_2).
\end{equation}
Since $l\geq m_2$ and $m_1\geq m_2$, we have $l'\geq 0$ and
$m_1'\geq 0$. Consequently, exactly one of the following formal
reduction cases occurs:
\begin{enumerate}[label=(\Roman*)]
\item $m_1'>0$ and
$l'\geq\min\{m_1',m_2\}$;
\item $m_1'>0$ and
$0<l'<\min\{m_1',m_2\}$;
\item $m_1'>0$ and $l'=0$;
\item $m_1'=0$.
\end{enumerate}

In the first two formal reduction cases, $m_1'>0$ and $l'>0$.
Equivalently, $m_1>m_2$ and $l>m_2$, so these are precisely the
formal cases lying in the strict reduction regime. In the third case,
$l'=0$, equivalently $l=m_2$ with $m_1>m_2$. This is the
unequal-length equality case, where $l'=0$ is only a formal value: the 
transformed element $\gamma_1\gamma_2^{-1}$ need not be hyperbolic, 
and its type depends on an additional overlap invariant. In the fourth case, 
$m_1'=0$, equivalently $m_1=m_2$.
The formal subtraction then predicts zero translation length for the
first Nielsen-transformed element. Hence this case is not part of the
strict hyperbolic reduction regime; it is the terminal equal-length
case.

In formal reduction cases (I) and (II), equivalently when
$m_1>m_2$ and $l>m_2$, the formal reduction data are the actual
geometric triple of the Nielsen-transformed pair. The
translation-length difference underlying this reduction appears in
Case (2)(ii) of \cite[Proposition 3.5]{Co}.
The next proposition identifies the specific Nielsen-transformed
element realizing this difference and determines its precise new
axis overlap and distinguished endpoint.

\begin{prop}[Strict Nielsen reduction]
\label{prop:reduce_ii}
Suppose that the geometric triple $(l,m_1,m_2)$ is in case (I)
and satisfies
\begin{equation}
\label{eq:strict_reduction_conditions}
m_1>m_2
\qquad\text{and}\qquad
l>m_2.
\end{equation}
Equivalently, its formal reduction data $(l',m_1',m_2)$ are in
formal reduction case (I) or (II). Then
$\gamma_1\gamma_2^{-1}$ is hyperbolic, with translation length
\begin{equation}
\label{eq:reduced_translation_length}
m_{\gamma_1\gamma_2^{-1}}=m_1-m_2=m_1'.
\end{equation}
Moreover,
\begin{equation}
\label{eq:reduced_axis_intersection_segment}
A_{\gamma_1\gamma_2^{-1}}\cap A_{\gamma_2}
=
[\gamma_2.v^-,v^+].
\end{equation}
Consequently,
\begin{equation}
\label{eq:reduced_axis_overlap}
\ell\bigl(
A_{\gamma_1\gamma_2^{-1}}\cap A_{\gamma_2}
\bigr)
=l-m_2=l'.
\end{equation}
Thus $(l',m_1',m_2)$ is the geometric triple of the ordered
generating pair
$(\gamma_1\gamma_2^{-1},\gamma_2)$. The translation directions of
$\gamma_1\gamma_2^{-1}$ and $\gamma_2$ induce the same orientation
on their common segment, namely from $\gamma_2.v^-$ toward $v^+$.
In particular, the strict Nielsen reduction preserves the common
orientation and the distinguished endpoint $v^+$.
\end{prop}

\begin{proof}

\begin{figure}[H]
\centering
\includegraphics[scale=1.25]{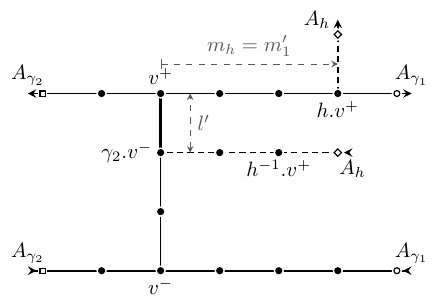}
\caption{Generator reduction from the geometric triple $(l,m_1,m_2)=(3,5,2)$ to $(l',m_1',m_2)=(1,3,2)$. The element $h = \gamma_1\gamma_2^{-1}$ is elliptic with $m_h = m_1'$. This is a reduction to case (II); the case (I) configuration is analogous.}
\label{fig:reduce_ii}
\end{figure}

As in Section~\ref{subsec:nielsen}, let the abstract free group
$F_2$ act on $X$ through the homomorphism sending $x_i$ to
$\gamma_i$. This action need not be faithful. Under the strict
reduction conditions, \cite[Proposition~4.1]{CuVo1}
implies that $\gamma_2^{-1}\gamma_1$ is hyperbolic with translation
length $m_1-m_2$. Since $\gamma_1\gamma_2^{-1}$ is conjugate to
$\gamma_2^{-1}\gamma_1$ by $\gamma_2$, it is also hyperbolic
with the same translation length. It remains to determine the
intersection of $A_{\gamma_2}$ and $A_{\gamma_1\gamma_2^{-1}}$.

Set $h\coloneqq\gamma_1\gamma_2^{-1}$ and $x=v^+$. Because the
two generators translate in the same direction along their common
segment and $m_1>m_2$, the points occur on $A_{\gamma_1}$ in
the order
\begin{equation}
\label{eq:strict_reduction_point_order}
\gamma_1^{-1}.x,\qquad
\gamma_2^{-1}.x,\qquad
x.
\end{equation}
Hence
\begin{equation}
\label{eq:dist_two_points}
\begin{split}
d(x,h.x)
&=d(\gamma_1^{-1}.x,\gamma_2^{-1}.x)\\
&=d(x,\gamma_1^{-1}.x)
-d(x,\gamma_2^{-1}.x)\\
&=m_1-m_2=m_h.
\end{split}
\end{equation}
Thus $x\in A_h$. Moreover,
$\gamma_2^{-1}.x\in[v^-,x]\subset A_{\gamma_1}$, so
$h.x\in A_{\gamma_1}$, and $[x,h.x]$ leaves $x$ in the positive
direction along $A_{\gamma_1}$. Since $x$ is the terminal point
of $A_{\gamma_1}\cap A_{\gamma_2}$ in that direction, the
positive germs of $A_h$ and $A_{\gamma_2}$ at $x$ are distinct.
Therefore, $A_h\cap A_{\gamma_2}$ cannot extend beyond $x$ in
the positive $A_{\gamma_2}$-direction.

We claim that $[\gamma_2.v^-,v^+]$ is contained in $A_h$.
Let $y$ be a vertex of $[\gamma_2.v^-,v^+]$. Then
$\gamma_2^{-1}.y$ lies in
$[v^-,\gamma_2^{-1}.v^+]$, which is contained in
$A_{\gamma_1}\cap A_{\gamma_2}$. Hence both
$\gamma_1^{-1}.y$ and $\gamma_2^{-1}.y$ lie on
$A_{\gamma_1}$. Since the two generators translate in the same
direction on their common segment, the points occur on
$A_{\gamma_1}$ in the order
$\gamma_1^{-1}.y$, $\gamma_2^{-1}.y$, $y$. Hence we have
\begin{equation}
\label{eq:reduced_axis_segment_minimal_displacement}
\begin{split}
d(y,h.y)
&=
d(\gamma_1^{-1}.y,\gamma_2^{-1}.y)\\
&=
m_1-m_2\\
&=
m_h.
\end{split}
\end{equation}
Thus every vertex of $[\gamma_2.v^-,v^+]$ lies on $A_h$.
Since $[\gamma_2.v^-,v^+]$ is also contained in
$A_{\gamma_2}$, we obtain
\begin{equation}
\label{eq:into_ping_pong}
[\gamma_2.v^-,v^+]
\subseteq
A_h\cap A_{\gamma_2}.
\end{equation}

We now show that the intersection cannot extend beyond
$\gamma_2.v^-$. Let $z$ be the vertex of $A_{\gamma_2}$ adjacent
to $\gamma_2.v^-$ in the direction of $v^-$. The point
$\gamma_1^{-1}.z$ lies on the negative ray of $A_{\gamma_1}$
issuing from $v^-$, while $\gamma_2^{-1}.z$ lies on the
negative ray of $A_{\gamma_2}$ issuing from $v^-$. By the
maximality of $A_{\gamma_1}\cap A_{\gamma_2}$, these two rays
have distinct germs at $v^-$. Consequently, the geodesic joining
$\gamma_1^{-1}.z$ to $\gamma_2^{-1}.z$ passes through $v^-$, and
\begin{equation}
\label{eq:reduced_negative_endpoint_displacement}
\begin{split}
d(z,h.z)
&=
d(\gamma_1^{-1}.z,\gamma_2^{-1}.z)\\
&=
d(\gamma_1^{-1}.z,v^-)+d(v^-,\gamma_2^{-1}.z)\\
&=
\bigl(m_1-m_2+1\bigr)+1\\
&=
m_1-m_2+2>m_h.
\end{split}
\end{equation}
Thus $z\notin A_h$. Since the intersection of two geodesics in a
tree is convex, it cannot extend beyond $\gamma_2.v^-$. Together
with the preceding inclusion and the exclusion beyond $x$, this
proves (\ref{eq:reduced_axis_intersection_segment}).

Since $h$ is hyperbolic, it acts on $A_h$ as a translation in
the direction from $x$ toward $h.x$. Because $[x,h.x]$ leaves
$x$ in the positive direction, this orientation restricts on
$[\gamma_2.v^-,x]$ as the orientation from $\gamma_2.v^-$ toward
$x$. By the initial normalization, the translation direction of
$\gamma_2$ induces the same orientation on this segment.
Therefore, $h$ and $\gamma_2$ induce the same orientation on
their common segment. Thus $(l',m_1',m_2)$ is the geometric
triple of $(\gamma_1\gamma_2^{-1},\gamma_2)$, with the common
orientation and the distinguished endpoint $v^+$ preserved.

\end{proof}

\begin{exm}[Running example]
\label{ex:running-example}
The preceding proposition can be iterated before a positive short
overlap is encountered. We record a running example that will be
used here to illustrate the local geometric reduction.

Assume that $X$ is locally finite, and consider a normalized ordered
hyperbolic pair with geometric triple
\begin{equation}
\label{eq:running_example_initial}
(l,m_1,m_2)=(14,13,5).
\end{equation}
Repeated application of Proposition~\ref{prop:reduce_ii} gives the
sequence of normalized geometric triples
\begin{equation}
\label{eq:running_example_reduction}
(14,13,5)
\longrightarrow
(9,8,5)
\longrightarrow
(4,5,3)
\longrightarrow
(1,3,2)
\end{equation}
The first arrow consists of one application of
Proposition~\ref{prop:reduce_ii}. For each of the second and third
arrows, one first applies Proposition~\ref{prop:reduce_ii} and then
interchanges the two generators to restore the nonincreasing order
of their translation lengths. Thus every displayed triple is
the actual geometric triple of the corresponding normalized ordered
pair. Proposition~\ref{prop:reduce_ii} preserves the common
orientation and the distinguished endpoint $v^+$ under each
subtraction, while the ensuing interchanges do not change the common
axis segment or its endpoint $v^+$.

At the final displayed state, the overlap has length $1$, which is
strictly smaller than the shorter translation length $2$. Hence
Proposition~\ref{prop:schottky} applies, and the ordered pair at this
first stopping state is Schottky.
\end{exm}

\subsection{Unequal-length Equality Reduction}
\label{se:exc_eq_redu}

The boundary case of formal reduction case (III), where
$l'=0$ and $m_1'>0$, differs from the strict reduction considered in
Proposition~\ref{prop:reduce_ii}. In this case, the numerical data
$(l',m_1',m_2)$ do not by themselves determine the type of the
Nielsen-transformed element. The outcome depends on an additional
overlap invariant involving a translate of $A_{\gamma_1}$. The
translation-length threshold underlying this boundary configuration
appears, after conjugating the configuration, in
\cite[Proposition 3.5, Case (2)(iii)]{Co}. The next
proposition identifies the specific Nielsen-transformed element,
determines whether it is elliptic or hyperbolic, and computes the
corresponding characteristic-set distances.

\begin{prop}[Unequal-length equality reduction]
\label{prop:reduce_iii}
Suppose that the geometric triple $(l,m_1,m_2)$ is in case (I)
and satisfies
\begin{equation}
\label{eq:reduce_iii_boundary_conditions}
m_1>m_2,
\qquad
l=m_2.
\end{equation}
That is, its formal reduction data
$(l',m_1',m_2)=(l-m_2,m_1-m_2,m_2)$
are in formal reduction case (III).

Set
\begin{equation}
\label{eq:reduce_iii_element}
h\coloneqq\gamma_1\gamma_2^{-1}.
\end{equation}
Define
\begin{equation}
\label{eq:reduce_iii_auxiliary_overlap}
l_0\coloneqq
\ell\bigl(A_{\gamma_1}\cap\gamma_2.A_{\gamma_1}\bigr),
\end{equation}
allowing $l_0=+\infty$ if the intersection contains an infinite ray.
Then exactly one of the following alternatives occurs.
\begin{itemize}
\item If
\begin{equation}
\label{eq:reduce_iii_elliptic_condition}
l_0\geq\frac{m_1-m_2}{2},
\end{equation}
then $h$ is elliptic. Moreover, $A_{\gamma_2}$ and $T_h$ are
disjoint, and
\begin{equation}
\label{eq:reduce_iii_axis_tree_distance}
d(A_{\gamma_2},T_h)
=
\frac{m_1-m_2}{2}.
\end{equation}

\item If
\begin{equation}
\label{eq:reduce_iii_hyperbolic_condition}
l_0<\frac{m_1-m_2}{2},
\end{equation}
then $h$ is hyperbolic, with translation length
\begin{equation}
\label{eq:reduce_iii_translation_length}
m_h=m_1-m_2-2l_0.
\end{equation}
The axes $A_h$ and $A_{\gamma_2}$ are edge-disjoint, and
\begin{equation}
\label{eq:reduce_iii_axis_distance}
d(A_{\gamma_2},A_h)=l_0.
\end{equation}
Consequently, the actual geometric triple of the ordered pair
$(h,\gamma_2)$ is
\begin{equation}
\label{eq:reduce_iii_geometric_triple}
(0,m_1-m_2-2l_0,m_2).
\end{equation}
\end{itemize}
\end{prop}

\begin{proof}
Set
\begin{equation}
\label{eq:reduce_iii_proof_notation}
h\coloneqq\gamma_1\gamma_2^{-1},
\qquad
x\coloneqq v^+,
\qquad
d\coloneqq m_1-m_2.
\end{equation}
Since $l'=0$, we have $l=m_2$. Hence
$\gamma_2^{-1}.x=v^-$, and the path used in the proof of
Proposition~\ref{prop:reduce_ii} is
\begin{equation}
\label{eq:reduce_iii_path_P}
P=[\gamma_1^{-1}.x,\gamma_2^{-1}.x]
=[\gamma_1^{-1}.x,v^-].
\end{equation}
Its length is $d$. Applying $\gamma_1$ and $\gamma_2$ gives
\begin{equation}
\label{eq:reduce_iii_image_paths}
\gamma_1.P=[x,h.x],
\qquad
\gamma_2.P=[h^{-1}.x,x].
\end{equation}
Both paths have length $d$:
\begin{equation}
\label{eq:reduce_iii_image_path_lengths}
\ell(\gamma_1.P)=\ell(\gamma_2.P)=d=m_1-m_2.
\end{equation}

We now compute $d(h.x,h^{-1}.x)$ by determining the intersection
of the two paths $\gamma_1.P=[x,h.x]$ and
$\gamma_2.P=[x,h^{-1}.x]$. The path $\gamma_1.P$ lies in
$A_{\gamma_1}$, whereas $\gamma_2.P$ lies in
$\gamma_2.A_{\gamma_1}$. Moreover,
$x\in A_{\gamma_1}\cap\gamma_2.A_{\gamma_1}$, because
$\gamma_2^{-1}.x=v^-\in A_{\gamma_1}$. Therefore, we first
determine where the overlap
$A_{\gamma_1}\cap\gamma_2.A_{\gamma_1}$ lies relative to $x$.

On the one hand, the edge of $\gamma_2.A_{\gamma_1}$ incident
to $x$ and contained in the path $[x,\gamma_2.x]$ is not
contained in $A_{\gamma_1}$. Indeed,
$[x,\gamma_2.x]\subset A_{\gamma_2}$, and its edge incident
to $x$ lies outside $[v^-,v^+]$ by the maximality of
$A_{\gamma_1}\cap A_{\gamma_2}=[v^-,v^+]$.

On the other hand, the other edge of
$\gamma_2.A_{\gamma_1}$ incident to $x$ does not leave $x$
along $A_{\gamma_1}$ toward $v^-$. Indeed, the edge is the
image under $\gamma_2$ of the edge of $A_{\gamma_1}$ incident
to $v^-$ and contained in the path
$[v^-,\gamma_1^{-1}.x]$. By the maximality of
$A_{\gamma_1}\cap A_{\gamma_2}=[v^-,v^+]$, that edge is
distinct from the edge of $A_{\gamma_2}$ incident to $v^-$
and contained in the path $[v^-,\gamma_2^{-1}.v^-]$.
Applying $\gamma_2$, we obtain that the edge of $\gamma_2.P$
incident to $x$ is distinct from the common edge of
$A_{\gamma_1}$ and $A_{\gamma_2}$ at $x$ directed toward
$v^-$.

Consequently, if
$A_{\gamma_1}\cap\gamma_2.A_{\gamma_1}$ contains an edge
incident to $x$, then this edge is the edge of $A_{\gamma_1}$
incident to $x$ and contained in $[x,h.x]$. This is precisely
the edge of $\gamma_1.P$ incident to $x$. Since both image
paths have length $d$, the definition of $l_0$ gives
\begin{equation}
\label{eq:reduce_iii_image_path_overlap}
\ell(\gamma_1.P\cap\gamma_2.P)
=
\min\{l_0,d\}.
\end{equation}

The two segments $\gamma_1.P=[x,h.x]$ and
$\gamma_2.P=[x,h^{-1}.x]$ have the common endpoint $x$ and
the same length $d$. Their intersection is a possibly
degenerate common subsegment containing $x$. Removing this
common subsegment from both segments gives
\begin{equation}
\label{eq:reduce_iii_inverse_orbit_distance}
\begin{split}
d(h.x,h^{-1}.x)
&=
2\bigl(
d-\ell(\gamma_1.P\cap\gamma_2.P)
\bigr)\\
&=
2\max\{0,d-l_0\}\\
&=
2\max\{0,m_1-m_2-l_0\}.
\end{split}
\end{equation}

According to Corollary~\ref{cor:aut_type}, the auxiliary length
$l_0$ decides the type of $h$. If
$l_0\geq\frac{m_1-m_2}{2}$,
\begin{equation}
\label{eq:reduce_iii_elliptic_distance_bound}
\begin{split}
d(h.x,h^{-1}.x)
&=
2\max\{0,m_1-m_2-l_0\}\\
&\leq m_1-m_2\\
&=
d(x,h.x),
\end{split}
\end{equation}
and thus $h$ is elliptic.

If $l_0<\frac{m_1-m_2}{2}$,
\begin{equation}
\label{eq:reduce_iii_hyperbolic_distance_bound}
\begin{split}
d(h.x,h^{-1}.x)
&=2(m_1-m_2-l_0)\\
&>m_1-m_2\\
&=d(x,h.x),
\end{split}
\end{equation}
and thus $h$ is hyperbolic, with translation length
\begin{equation}
\label{eq:reduce_iii_hyperbolic_length_computation}
\begin{split}
m_h
&=
d(h.x,h^{-1}.x)-d(x,h.x)\\
&=
2(m_1-m_2-l_0)-(m_1-m_2)\\
&=
m_1-m_2-2l_0
\end{split}
\end{equation}
given by Corollary~\ref{cor:aut_impl}.

These give the two alternatives in the statement:

\textbf{Case (i): $h$ is elliptic.}

\begin{figure}[H]
\centering
\includegraphics[scale=1.2]{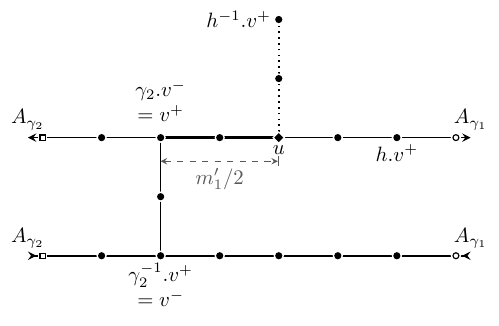}
\caption{An elliptic reduction from
$(l,m_1,m_2)=(2,6,2)$, with $l_0=2$. The element $h=\gamma_1\gamma_2^{-1}$ is elliptic, and the fixed vertex $u$ is marked.}
\label{fig:reduce_iiib}
\end{figure}

Corollary~\ref{cor:aut_impl} further implies that the segment
$[x,h.x]$ meets $T_h$ at its midpoint, which we denote by $u$.
By the elliptic displacement formula, this midpoint is a vertex,
and
\begin{equation}
\label{eq:reduce_iii_midpoint_distance}
d(x,u)=\frac{1}{2}d(x,h.x)=\frac{m_1-m_2}{2}.
\end{equation}

It remains to show that no point of $T_h$ lies closer to
$A_{\gamma_2}$. Let $y$ be a vertex of $A_{\gamma_2}$. By
the reverse triangle inequality and the displacement formulas
for $\gamma_1$ and $\gamma_2$, we have
\begin{equation}
\label{eq:reduce_iii_axis_displacement_lower_bound}
\begin{split}
d(y,h.y)
&=
d(\gamma_1^{-1}.y,\gamma_2^{-1}.y)\\
&\geq
d(y,\gamma_1^{-1}.y)
-
d(y,\gamma_2^{-1}.y)\\
&=
m_1+2d(y,A_{\gamma_1})-m_2\\
&\geq m_1-m_2.
\end{split}
\end{equation}
Since $h$ is elliptic, Proposition~\ref{prop:elliptic} gives
\begin{equation}
\label{eq:reduce_iii_fixed_tree_distance_lower_bound}
d(y,T_h)
=
\frac{1}{2}d(y,h.y)
\geq
\frac{m_1-m_2}{2}.
\end{equation}
Hence
\begin{equation}
d(A_{\gamma_2},T_h)
\geq
\frac{m_1-m_2}{2}.
\end{equation}
On the other hand, $x\in A_{\gamma_2}$ and $u\in T_h$, so
(\ref{eq:reduce_iii_midpoint_distance}) gives
\begin{equation}
d(A_{\gamma_2},T_h)
\leq
d(x,u)=\frac{m_1-m_2}{2}.
\end{equation}
Therefore,
\begin{equation}
\label{eq:reduce_iii_elliptic_axis_tree_distance}
d(A_{\gamma_2},T_h)=\frac{m_1-m_2}{2}.
\end{equation}
Since $m_1-m_2>0$, the two characteristic sets are disjoint.

\textbf{Case (ii): $h$ is hyperbolic.}

\begin{figure}[H]
\centering
\includegraphics[scale=1.2]{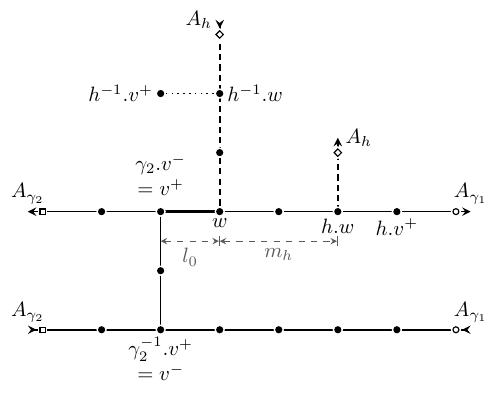}
\caption{A hyperbolic reduction from
$(l,m_1,m_2)=(2,6,2)$, with $l_0=1$. The element
$h = \gamma_1\gamma_2^{-1}$ is hyperbolic, and its axis is
edge-disjoint from $A_{\gamma_2}$.}
\label{fig:reduce_iiia}
\end{figure}

We now determine the relative position of $A_{\gamma_2}$ and
$A_h$. Let $w$ be the projection of $x=v^+$ onto $A_h$. By
Corollary~\ref{cor:aut_impl},
\begin{equation}
\label{eq:reduce_iii_projection_distance}
\begin{split}
d(x,w)
&=
d(x,h.x)
-
\frac{1}{2}d(h.x,h^{-1}.x)\\
&=
(m_1-m_2)
-
(m_1-m_2-l_0)\\
&=
l_0.
\end{split}
\end{equation}

We first show that the edge of $[x,h.x]$ incident to $x$ is not
contained in $A_{\gamma_2}$, and that the edge of
$[x,h^{-1}.x]$ incident to $x$ is not contained in
$A_{\gamma_2}$. The edge of $[x,h.x]=\gamma_1.P$ incident to
$x$ is the edge of $A_{\gamma_1}$ incident to $x$ and
contained in $[x,h.x]$. This edge is not contained in
$A_{\gamma_2}$, because it is not contained in $[v^-,v^+]$,
whereas $A_{\gamma_1}\cap A_{\gamma_2}=[v^-,v^+]$.

For the second segment, the edge of $[x,h^{-1}.x]$ incident
to $x$ is the edge of $\gamma_2.P$ incident to $x$. It is
the image under $\gamma_2$ of the edge of $A_{\gamma_1}$
incident to $v^-$ and contained in
$[v^-,\gamma_1^{-1}.x]$. This image is not the edge of
$A_{\gamma_2}$ incident to $x$ and contained in $[x,v^-]$.
Indeed, the preimage of that edge under $\gamma_2$ is the
edge of $A_{\gamma_2}$ incident to $v^-$ and contained in
$[v^-,\gamma_2^{-1}.v^-]$, and this edge is distinct from
the edge of $A_{\gamma_1}$ incident to $v^-$ and contained
in $[v^-,\gamma_1^{-1}.x]$ by the maximality of
$A_{\gamma_1}\cap A_{\gamma_2}$.

The same image is also not the edge of $A_{\gamma_2}$
incident to $x$ and contained in $[x,\gamma_2.x]$. The
preimage of that edge under $\gamma_2$ is the edge of
$A_{\gamma_2}$ incident to $v^-$ and contained in $[v^-,x]$,
whereas the edge of $A_{\gamma_1}$ under consideration is
contained in $[v^-,\gamma_1^{-1}.x]$. Hence neither
$[x,h.x]$ nor $[x,h^{-1}.x]$ has its edge incident to $x$
contained in $A_{\gamma_2}$.

Suppose first that $l_0>0$. Since $w$ is the projection of
$x$ onto $A_h$, we have
$[x,w]\subseteq[x,h.x]\subset A_{\gamma_1}$. The segment
$[x,w]$ leaves $x=v^+$ in the positive direction along
$A_{\gamma_1}$. Since $v^+$ is the positive endpoint of
$A_{\gamma_1}\cap A_{\gamma_2}$, it follows that
$[x,w]\cap A_{\gamma_2}=\{x\}$. Thus $x$ is the projection
of $w$ onto $A_{\gamma_2}$. Since $w$ is also the projection
of $x$ onto $A_h$, the segment $[x,w]$ is the unique bridge
between the two axes. Consequently,
\begin{equation}
\label{eq:reduce_iii_positive_axis_distance}
d(A_{\gamma_2},A_h)=d(x,w)=l_0.
\end{equation}

Suppose now that $l_0=0$. Then
\begin{equation}
\label{eq:reduce_iii_zero_projection}
w=x=v^+.
\end{equation}
Since $x\in A_h$, the two segments from $x$ to its forward
and backward translates lie on $A_h$:
\begin{equation}
\label{eq:reduce_iii_axis_segments}
[x,h.x]\cup[x,h^{-1}.x]
\subseteq A_h.
\end{equation}
Moreover, the two edges of $A_h$ incident to $x$ are precisely
the edge of $[x,h.x]$ incident to $x$ and the edge of
$[x,h^{-1}.x]$ incident to $x$. By the preceding edge
analysis, neither belongs to $A_{\gamma_2}$. Convexity of
the two axes therefore gives
\begin{equation}
\label{eq:reduce_iii_single_vertex_intersection}
A_{\gamma_2}\cap A_h
=
\{v^+\}.
\end{equation}
Thus the axes are edge-disjoint and
\begin{equation}
\label{eq:reduce_iii_zero_axis_distance}
d(A_{\gamma_2},A_h)=0=l_0.
\end{equation}

In both subcases of Case (ii), the axes are edge-disjoint.
By (\ref{eq:reduce_iii_positive_axis_distance}) in the
subcase $l_0>0$ and by
(\ref{eq:reduce_iii_zero_axis_distance}) in the subcase
$l_0=0$, their distance is $l_0$. Together with
(\ref{eq:reduce_iii_hyperbolic_length_computation}), this
shows that the geometric triple of $(h,\gamma_2)$ is
\begin{equation}
\label{eq:reduce_iii_concluded_geometric_triple}
(0,m_1-m_2-2l_0,m_2).
\end{equation}
This completes the proof.
\end{proof}

\begin{rmk}[The auxiliary overlap]
The auxiliary intersection
$A_{\gamma_1}\cap\gamma_2.A_{\gamma_1}$ is not assumed to be
finite. Accordingly, in Proposition~\ref{prop:reduce_iii} we allow
$l_0=+\infty$ when this intersection contains an infinite ray.
The condition $l_0<(m_1-m_2)/2$ in the hyperbolic alternative
automatically forces $l_0<\infty$. In the elliptic alternative,
$l_0$ may be infinite, but its precise value does not affect
the characteristic-set distance. In the proof,
$\min\{l_0,d\}$ is understood in the extended sense.
\end{rmk}

\subsection{The Equal-Length Terminal Case}
\label{se:eq_length_case}

The remaining formal reduction case occurs when the translation lengths
of the current normalized pair are equal. In this situation, the
formal reduced translation length $m_1'$ is zero. The proposition
below shows that the corresponding Nielsen-transformed element is
elliptic, and its fixed subtree meets $A_{\gamma_2}$ precisely
in the residual segment measured by the formal overlap length.

\begin{prop}[Equal-length terminal reduction]
\label{prop:reduce_iv}
Suppose that the geometric triple $(l,m_1,m_2)$ is in case (I)
and satisfies
\begin{equation}
\label{eq:reduce_iv_boundary_conditions}
m_1=m_2,
\qquad
l\geq m_2.
\end{equation}
That is, its formal reduction data
$(l',m_1',m_2)=(l-m_2,m_1-m_2,m_2)$
are in formal reduction case (IV).

Then $\gamma_1\gamma_2^{-1}$ is elliptic. Moreover,
\begin{equation}
\label{eq:reduce_iv_statement_intersection}
A_{\gamma_2}\cap T_{\gamma_1\gamma_2^{-1}}
=
[\gamma_2.v^-,v^+].
\end{equation}
Consequently,
\begin{equation}
\label{eq:reduce_iv_axis_tree_intersection_length}
\ell\bigl(
A_{\gamma_2}\cap T_{\gamma_1\gamma_2^{-1}}
\bigr)
=
l'=l-m_2.
\end{equation}
\end{prop}

\begin{proof}
Set
\begin{equation}
\label{eq:reduce_iv_element}
h\coloneqq\gamma_1\gamma_2^{-1}.
\end{equation}
We first prove that $h$ is elliptic. Since
$l\geq m_1=m_2$, the points $\gamma_1^{-1}.v^+$ and
$\gamma_2^{-1}.v^+$ both lie in
$A_{\gamma_1}\cap A_{\gamma_2}$. Moreover,
\begin{equation}
\label{eq:reduce_iv_equal_preimage_distances}
d(\gamma_1^{-1}.v^+,v^+)
=
m_1=m_2
=
d(\gamma_2^{-1}.v^+,v^+).
\end{equation}
Since $\gamma_1$ and $\gamma_2$ translate in the same direction
along their common axis segment, the two points lie on the same
geodesic ray from $v^+$. Their equal distances from $v^+$ therefore
give
\begin{equation}
\label{eq:reduce_iv_equal_preimages}
\gamma_1^{-1}.v^+=\gamma_2^{-1}.v^+.
\end{equation}
Applying $\gamma_1$ to
(\ref{eq:reduce_iv_equal_preimages}) yields
\begin{equation}
\label{eq:reduce_iv_fixed_endpoint}
h.v^+=v^+.
\end{equation}
Thus $h$ fixes $v^+$ and is elliptic.

\begin{figure}[H]
\centering
\includegraphics[scale=1.25]{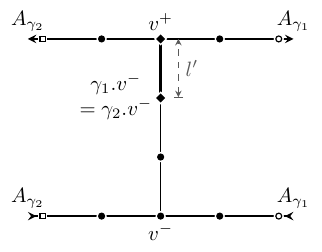}
\caption{An elliptic reduction from
$(l,m_1,m_2)=(3,2,2)$. The element
$h = \gamma_1\gamma_2^{-1}$ is elliptic. The vertices of the fixed segment $[\gamma_2.v^-,v^+]$ are marked.}
\label{fig:reduce_iv}
\end{figure}

We now determine the intersection of $A_{\gamma_2}$ with $T_h$.
First, we show that the segment from $\gamma_2.v^-$ to $v^+$
is fixed pointwise by $h$. Set
\begin{equation}
\label{eq:reduce_iv_candidate_fixed_segment}
I\coloneqq[\gamma_2.v^-,v^+].
\end{equation}
Since $m_1=m_2$ and the generators translate in the same
direction along their common segment, we have
\begin{equation}
\label{eq:reduce_iv_equal_length_overlap}
\gamma_1.v^-=\gamma_2.v^-.
\end{equation}
Consequently,
\begin{equation}
\label{eq:reduce_iv_equal_action_on_overlap}
h.(\gamma_2.v^-)
=
\gamma_1.v^-
=
\gamma_2.v^-.
\end{equation}
Thus $h$ fixes both endpoints of $I$. Since an automorphism
of a tree fixing two vertices fixes the geodesic joining them
pointwise, $h$ fixes $I$ pointwise. Therefore,
\begin{equation}
\label{eq:reduce_iv_fixed_segment_inclusion}
[\gamma_2.v^-,v^+]
\subseteq
A_{\gamma_2}\cap T_h.
\end{equation}

It remains to show that this intersection does not extend
beyond either endpoint.

First, let $e^+$ be the edge of
$A_{\gamma_2}\setminus A_{\gamma_1}$ incident with $v^+$.
We claim that
\begin{equation}
\label{eq:reduce_iv_positive_edge_preimages}
\gamma_1^{-1}.e^+
\neq
\gamma_2^{-1}.e^+.
\end{equation}
Indeed, since $l\geq m_2$, the edge
$\gamma_2^{-1}.e^+$ is contained in $[v^-,v^+]$,
and hence in $A_{\gamma_1}$. If the two edges were
equal, applying $\gamma_1$ would imply that
$e^+$ lies in $A_{\gamma_1}$, contrary to its
definition.

Applying $\gamma_1$ to
(\ref{eq:reduce_iv_positive_edge_preimages}) gives
\begin{equation}
\label{eq:reduce_iv_positive_edge_not_fixed}
h.e^+
=
\gamma_1\gamma_2^{-1}.e^+
\neq e^+.
\end{equation}
If $A_{\gamma_2}\cap T_h$ extended beyond $v^+$,
then the convexity of $T_h$, together with
(\ref{eq:reduce_iv_fixed_segment_inclusion}), would
imply that $e^+\subset T_h$. Hence both endpoints
of $e^+$ would be fixed by $h$, so $h.e^+=e^+$.
This contradicts
(\ref{eq:reduce_iv_positive_edge_not_fixed}).
Therefore, the intersection
$A_{\gamma_2}\cap T_h$ does not extend beyond $v^+$.

At the other endpoint, recall that
\begin{equation}
\label{eq:reduce_iv_equal_images_of_v_minus}
\gamma_1.v^-=\gamma_2.v^-.
\end{equation}
Let $e^-$ be the edge of
$A_{\gamma_2}\setminus A_{\gamma_1}$ incident
with $v^-$. Then $\gamma_2.e^-$ is the edge
of $A_{\gamma_2}$ incident with $\gamma_2.v^-$
and contained in $[v^-,\gamma_2.v^-]$.
We claim that
\begin{equation}
\label{eq:reduce_iv_negative_edge_images}
\gamma_1.e^-\neq\gamma_2.e^-,
\qquad
\gamma_1.e^-\cap\gamma_2.e^-
=
\{\gamma_2.v^-\}.
\end{equation}
Indeed, since $l\geq m_2$, the edge
$\gamma_2.e^-$ is contained in
$A_{\gamma_1}\cap A_{\gamma_2}$.
If the two edges were equal, applying
$\gamma_1^{-1}$ would imply that
$e^-\subset A_{\gamma_1}$, contrary to
the definition of $e^-$. Thus the two
edges are distinct and share the vertex
$\gamma_2.v^-$.

Since
\begin{equation}
\label{eq:reduce_iv_negative_edge_action}
h.(\gamma_2.e^-)=\gamma_1.e^-,
\end{equation}
the edge $\gamma_2.e^-$ is not fixed by
$h$, by (\ref{eq:reduce_iv_negative_edge_images}).
If $A_{\gamma_2}\cap T_h$ extended beyond
$\gamma_2.v^-$, then the convexity of $T_h$,
together with
(\ref{eq:reduce_iv_fixed_segment_inclusion}),
would imply that $\gamma_2.e^-\subset T_h$.
Hence both endpoints of $\gamma_2.e^-$
would be fixed by $h$, so
$h.(\gamma_2.e^-)=\gamma_2.e^-$.
This contradicts
(\ref{eq:reduce_iv_negative_edge_action}) and
(\ref{eq:reduce_iv_negative_edge_images}).
Therefore, the intersection does not extend
beyond $\gamma_2.v^-$.

It follows that
\begin{equation}
\label{eq:reduce_iv_exact_axis_tree_intersection}
A_{\gamma_2}\cap T_h
=
[\gamma_2.v^-,v^+].
\end{equation}
Consequently,
\begin{equation}
\label{eq:reduce_iv_intersection_length_computation}
\begin{split}
\ell(A_{\gamma_2}\cap T_h)
&=
d(\gamma_2.v^-,v^+)\\
&=
d(v^-,v^+)-d(v^-,\gamma_2.v^-)\\
&=
l-m_2\\
&=
l'.
\end{split}
\end{equation}
This completes the proof.
\end{proof}

Together, Proposition~\ref{prop:schottky},
Proposition~\ref{prop:reduce_ii}, Proposition~\ref{prop:reduce_iii},
and Proposition~\ref{prop:reduce_iv} establish the local alternatives
for an ordered hyperbolic pair in normalized position. Either the pair
satisfies the short-overlap condition, in which case it is Schottky
when $X$ is locally finite, or a strict Nielsen subtraction realizes
the formal reduction data, or one of the two terminal equality cases
occurs. The next section assembles these local alternatives through
the Euclidean division process.

\section{Euclidean Resolution of the Reduction Process}
\label{sec:euclidean-resolution}

Recall that $X$ is locally finite, as assumed throughout the paper.
In this section, we assemble the local geometric alternatives
established in Section~\ref{sec:schottky-criterion} into a finite
Euclidean resolution of the Nielsen reduction process. The strict case
realizes one step of Euclidean subtraction on the geometric triple,
subtracting the shorter translation length from both the longer
translation length and the axis-overlap length. The strict reduction
stops precisely when the current normalized triple enters one of the
terminal geometric regimes: short overlap, the unequal-length boundary
case, or the equal-length case.

In the short-overlap case, Proposition~\ref{prop:schottky} shows that
the stopping pair is already terminal. In either equality case,
Propositions~\ref{prop:reduce_iii} and~\ref{prop:reduce_iv} describe
the additional Nielsen move that produces the terminal pair. Thus,
the formal Euclidean evolution and the actual Nielsen-transformed
generators remain synchronized through the first geometric stopping
state. Any formal continuation beyond this state is not asserted to
describe the actual transformed geometry. The terminal
characteristic-set geometry is determined by the stopping regime,
together with the auxiliary overlap $l_0$ in the unequal-length
boundary case.

\subsection{Formal Euclidean evolution and geometric stopping}
\label{se:f_euc_evo_stop}

Let $(\gamma_1,\gamma_2)$ be a pair of hyperbolic automorphisms such
that $\Gamma\coloneqq\langle\gamma_1,\gamma_2\rangle$ is discrete.
Assume that the axes $A_{\gamma_1}$ and $A_{\gamma_2}$ have finite
edge intersection, and let $(l,m_1,m_2)$ be the geometric triple
of the ordered pair. After interchanging the generators if necessary,
assume that $m_1\geq m_2>0$. If $l>0$, invert one generator if
necessary so that the two generators translate in the same direction
on their common axis segment, oriented from $v^-$ toward $v^+$.
We retain the notation $(\gamma_1,\gamma_2)$ for the normalized pair.
This normalization preserves the generated subgroup, the unordered
pair of translation lengths, and the overlap length.

If $l=0$, no common orientation or distinguished endpoint is needed.
Proposition~\ref{prop:schottky} shows immediately that the pair is
Schottky, so the initial pair is already the first stopping state.
This case is nevertheless included in the formal Euclidean analysis
below.

We apply successive Nielsen transformations while tracking the
translation lengths and overlap data. In
\cite[Algorithm 4.1, steps (3)--(5)]{Co}, the two
candidate Nielsen products $XY$ and $X^{-1}Y$ are compared by
translation length, and a candidate realizing the minimum is used
in the reduction step when the stated threshold is satisfied.
In the present normalized geometry,
Proposition~\ref{prop:reduce_ii} identifies the exact product used
in every strict reduction and determines the corresponding overlap
data. The present section records this refined evolution and resolves
its stopping stage in Euclidean terms.

At each stage of the procedure, the notation
$(\gamma_1,\gamma_2)$ and $(l,m_1,m_2)$ is reset to denote the current
ordered pair and its current geometric triple. The reduction proceeds
by applying the first applicable rule in the following list.

\begin{itemize}

\item If $m_1<m_2$, replace the ordered pair
$(\gamma_1,\gamma_2)$ by $(\gamma_2,\gamma_1)$ and replace
$(l,m_1,m_2)$ by $(l,m_2,m_1)$.

\item If $m_1>m_2$ and $l>m_2$, replace
$(\gamma_1,\gamma_2)$ by
$(\gamma_1\gamma_2^{-1},\gamma_2)$ and, by
Proposition~\ref{prop:reduce_ii}, replace $(l,m_1,m_2)$ by
$(l-m_2,m_1-m_2,m_2)$.

\item If $m_1>m_2$ and $l=m_2$, terminate the strict reduction.
Proposition~\ref{prop:reduce_iii} determines whether the terminal
Nielsen-transformed element $\gamma_1\gamma_2^{-1}$ is elliptic or
hyperbolic from the additional invariant $l_0$. In the hyperbolic
alternative, the reduced pair has edge-disjoint axes, so it is
Schottky by Proposition~\ref{prop:schottky}.

\item If $m_1=m_2$ and $l\geq m_2$, terminate the strict reduction.
Proposition~\ref{prop:reduce_iv} shows that the terminal
Nielsen-transformed element $\gamma_1\gamma_2^{-1}$ is elliptic.

\item If $l<m_2$, terminate the reduction.
Proposition~\ref{prop:schottky} shows that the current pair is a
Schottky pair.

\end{itemize}

Interchanges only restore the required ordering of the translation
lengths and do not change their sum. They cannot occur twice
consecutively, so a nonterminal interchange is necessarily followed
by a strict reduction. A strict reduction replaces the two positive
translation lengths $m_1$ and $m_2$ by $m_1-m_2$ and $m_2$, so
$m_1+m_2$ strictly decreases at every strict reduction. Since the
translation lengths at all strict stages are positive integers,
only finitely many such stages can occur. As at most one interchange
occurs between consecutive strict reductions, the complete reduction
procedure terminates.

We next fix the Euclidean and Nielsen notation used throughout this
section. In contrast to the temporary current-stage notation used
above, the symbols $(\gamma_1,\gamma_2)$ and $(l,m_1,m_2)$ now refer
to the fixed normalized initial pair and its geometric triple.
Recall that
\begin{equation}
\label{eq:first_axis}
l=\ell\bigl(A_{\gamma_1}\cap A_{\gamma_2}\bigr),
\end{equation}
with the convention that this length is zero in the edge-disjoint
case. Define the quotient and remainder sequences recursively by
\begin{equation}
\label{eq:Euc_algo}
q_i=\left\lfloor\frac{m_i}{m_{i+1}}\right\rfloor,
\qquad
m_{i+2}=m_i-q_im_{i+1},
\qquad
i\geq 1,
\end{equation}
as long as $m_{i+1}>0$. Let $k$ be the terminal index with
\begin{equation}
\label{eq:last_euclidean_remainder}
m_k=\gcd(m_1,m_2),
\qquad
m_{k+1}=0,
\end{equation}
so \eqref{eq:Euc_algo} is defined for $1\leq i\leq k-1$.

The quotient $q_i$ records the number of formal subtraction steps
in the $i$-th Euclidean block, followed by interchange at the block
boundary when required. Compressing this block into a single
Nielsen word gives
\begin{equation}
\label{eq:group_ele_seq}
\gamma_{i+2}=\gamma_i\gamma_{i+1}^{-q_i},
\qquad
1\leq i\leq k-1.
\end{equation}
The associated compressed formal ordered pairs are
$(\gamma_i,\gamma_{i+1})$.

Independently of geometric stopping, the initial lengths $m_1$ and
$m_2$ determine the complete quotient sequence $q_i$ and remainder
sequence $m_i$ in \eqref{eq:Euc_algo}. The corresponding formal
Nielsen words are defined by \eqref{eq:group_ele_seq} through the
complete Euclidean process. The formal first coordinate is
initialized at $l$, decreases by the current second translation
length at each formal subtraction, and is unchanged under
interchange. Throughout the geometrically realized strict evolution,
Proposition~\ref{prop:reduce_ii} identifies this formal coordinate
with the actual overlap length and the formal Nielsen pair with
the realized ordered pair. After the first geometric stopping state,
any continuation of the three formal coordinates is formal only
and is not asserted to describe the actual transformed geometry.
In either equality case, the prescribed terminal Nielsen move
is analyzed separately by Proposition~\ref{prop:reduce_iii} or
Proposition~\ref{prop:reduce_iv}, as appropriate.

We now index the uncompressed normalized formal states as follows.
The initial state has index $0$. Each subsequent state is obtained
by one formal subtraction, followed, when necessary, by the
interchange that restores the nonincreasing order of the two
translation-length coordinates. The index $i$ labels the Euclidean
quotient blocks, whereas the normalized-state index $n$ counts
individual formal subtraction steps. The following lemma records
the invariant along this uncompressed formal evolution and the
possible first stopping states.

\begin{lem}[Formal Euclidean evolution and geometric stopping]
\label{lem:euc}
Let $N$ be the first index at which the strict geometric
reduction does not apply. For $0\leq n\leq N$, the normalized
formal state is realized by a hyperbolic pair. Write its
geometric triple as
\begin{equation}
\label{eq:euc-realized-triples}
\bigl(
l^{(n)},m_1^{(n)},m_2^{(n)}
\bigr),
\qquad
0\leq n\leq N,
\end{equation}
where
\begin{equation}
\label{eq:euc-initial-state}
(l^{(0)},m_1^{(0)},m_2^{(0)})
=
(l,m_1,m_2).
\end{equation}
Then
\begin{equation}
\label{eq:invariant}
l^{(n)}-m_1^{(n)}-m_2^{(n)}=l-m_1-m_2
\end{equation}
for every $0\leq n\leq N$.

At the first geometric stopping state, exactly one of the
following conditions holds:

\begin{enumerate}

\item[(i)]
\begin{equation}
\label{eq:euc-short-overlap-stop}
0\leq l^{(N)}<m_2^{(N)};
\end{equation}

\item[(ii)]
\begin{equation}
\label{eq:euc-endpoint-stop}
l^{(N)}=m_2^{(N)}
\quad\text{and}\quad
m_1^{(N)}>m_2^{(N)};
\end{equation}

\item[(iii)]
\begin{equation}
\label{eq:euc-equal-length-stop}
m_1^{(N)}=m_2^{(N)}
\quad\text{and}\quad
l^{(N)}\geq m_2^{(N)}.
\end{equation}

\end{enumerate}

Set $d=\gcd(m_1,m_2)$. The complete formal Euclidean
continuation ends at the ordered state
\begin{equation}
\label{eq:euc-formal-final-state}
\bigl(
l-m_1-m_2+d,
d,
0
\bigr).
\end{equation}
Its first coordinate is formal only and is not asserted to
represent an actual axis-overlap length.
\end{lem}

\begin{proof}
For each $1\leq i\leq k-1$, the Euclidean relation
\eqref{eq:Euc_algo} is
$m_i=q_i m_{i+1}+m_{i+2}$, where
$0\leq m_{i+2}<m_{i+1}$. For $0\leq r\leq q_i$,
after $r$ formal subtractions in the $i$-th block, the
ordered length pair before the possible block-end interchange is
\begin{equation}
\label{eq:euc-length-evolution}
\bigl(
m_i-rm_{i+1},
m_{i+1}
\bigr).
\end{equation}
At $r=q_i$, this becomes $(m_{i+2},m_{i+1})$. If
$m_{i+2}>0$, interchange gives $(m_{i+1},m_{i+2})$,
which begins the next block. At the end of the final block,
the ordered length pair is $(m_{k+1},m_k)=(0,m_k)$,
which becomes $(m_k,0)$ after a final interchange.

The corresponding formal Nielsen pair in the $i$-th block,
before the possible block-end interchange, is
\begin{equation}
\label{eq:gamma-sequence}
\bigl(
\gamma_i\gamma_{i+1}^{-r},
\gamma_{i+1}
\bigr),
\qquad
0\leq r\leq q_i.
\end{equation}
At the block end $r=q_i$, \eqref{eq:gamma-sequence} gives
\begin{equation}
\label{eq:gamma-sequence-recursion}
\bigl(
\gamma_i\gamma_{i+1}^{-q_i},
\gamma_{i+1}
\bigr)
=
(\gamma_{i+2},\gamma_{i+1}),
\end{equation}
which becomes $(\gamma_{i+1},\gamma_{i+2})$ after the
restoring interchange. If $m_{i+2}>0$, this begins
the next block. Hence \eqref{eq:euc-length-evolution}
and \eqref{eq:gamma-sequence} describe the complete
formal Euclidean and Nielsen evolutions.

We next identify the states realized geometrically. For every
$n<N$, the normalized state satisfies
$m_1^{(n)}>m_2^{(n)}>0$ and
$l^{(n)}>m_2^{(n)}$. Proposition~\ref{prop:reduce_ii}
applies and produces the pre-normalized geometric triple
\begin{equation}
\label{eq:euc-single-subtraction}
\bigl(
l^{(n)}-m_2^{(n)},
m_1^{(n)}-m_2^{(n)},
m_2^{(n)}
\bigr).
\end{equation}
The same proposition preserves the common orientation and the
distinguished endpoint $v^+$, and identifies the transformed
ordered Nielsen pair exactly. Interchange restores the length
order, preserves the overlap length and the common orientation,
and exchanges the two translation-length coordinates. Hence
no further inversion is required. Induction shows that every
formal state with index at most $N$ is realized by the
corresponding hyperbolic pair, with the stated geometric triple.
By the definition of $N$, the state with index $N$ is the
first geometric stopping state.

The invariant holds initially by \eqref{eq:euc-initial-state}.
Under the strict subtraction \eqref{eq:euc-single-subtraction},
its value is
\begin{equation}
\label{eq:euc-invariant-computation}
\begin{aligned}
&\bigl(l^{(n)}-m_2^{(n)}\bigr)
-\bigl(m_1^{(n)}-m_2^{(n)}\bigr)
-m_2^{(n)}\\
&\qquad=
l^{(n)}-m_1^{(n)}-m_2^{(n)}.
\end{aligned}
\end{equation}
Interchanging the two generators preserves the overlap length
and the sum of the two translation lengths. Induction therefore
proves \eqref{eq:invariant} at every realized state.

It remains to identify the first stopping state. At the
normalized state with index $N$, one has
$m_1^{(N)}\geq m_2^{(N)}>0$. If
$l^{(N)}<m_2^{(N)}$, then
\eqref{eq:euc-short-overlap-stop} holds, and
Proposition~\ref{prop:schottky} shows that the stopping
pair is Schottky.

Suppose that $l^{(N)}\geq m_2^{(N)}$. If
$m_1^{(N)}=m_2^{(N)}$, then
\eqref{eq:euc-equal-length-stop} holds, and
Proposition~\ref{prop:reduce_iv} gives the equal-length
terminal reduction. Suppose instead that
$m_1^{(N)}>m_2^{(N)}$. Since $N$ is the first stopping
index, the other strict reduction condition
$l^{(N)}>m_2^{(N)}$ cannot hold. Therefore,
$l^{(N)}=m_2^{(N)}$, and
\eqref{eq:euc-endpoint-stop} holds. At the state with
index $N$, Proposition~\ref{prop:reduce_iii} governs
whether the element produced by the additional Nielsen
move is elliptic or hyperbolic. In the hyperbolic case,
its translation length is determined by the Euclidean
reduction data together with the additional invariant $l_0$,
and the terminal pair has edge-disjoint axes. This
additional Nielsen move lies outside the strict evolution
considered in the lemma. The three stopping conditions
are exhaustive and mutually exclusive.

Finally, continue the formal first coordinate after index $N$
by the rule fixed before the lemma. The algebraic computation
in \eqref{eq:euc-invariant-computation} shows that every
formal subtraction preserves the invariant, and interchange
preserves it as well. At the final ordered length pair
$(m_k,0)$, the formal first coordinate is
$l-m_1-m_2+m_k=l-m_1-m_2+d$, which gives
\eqref{eq:euc-formal-final-state}.
\end{proof}

The invariant in (\ref{eq:invariant}) converts each geometric stopping
condition into a condition on the initial overlap $l$, as made explicit
below.

\begin{thm}[Euclidean localization of the stopping stage]
\label{thm:termination}
Let $(l,m_1,m_2)$ be the geometric triple of the normalized pair fixed
in Section~\ref{se:f_euc_evo_stop}, with $m_1\geq m_2>0$. Let
$\{m_i\}_{i=1}^{k+1}$ and $\{q_i\}_{i=1}^{k-1}$ be the sequences
defined in (\ref{eq:Euc_algo}), and set $d=\gcd(m_1,m_2)$.
The stopping state $N$ of the reduction, defined in
Lemma~\ref{lem:euc}, is determined as follows.

\begin{enumerate}[label=(\arabic*)]
\item If $l=0$, the initial pair is the first stopping state and
satisfies the zero-overlap stopping condition.

\item Suppose that
\begin{equation}
\label{eq:termination-positive-range}
0<l\leq m_1+m_2-d.
\end{equation}
There are unique integers $j$ and $q$ such that
\begin{equation}
\label{eq:l-inequality-main}
\begin{split}
(m_1+m_2)-m_j+(q-1)m_{j+1}
&<l\\
&\leq
(m_1+m_2)-m_j+qm_{j+1},
\end{split}
\end{equation}
where
\begin{equation}
\label{eq:termination-index-ranges}
1\leq j\leq k-1,
\qquad
0\leq q\leq q_j-1.
\end{equation}
The normalized pair at the stopping state $N$ is
\begin{equation}
\label{eq:termination-current-pair}
\bigl(
\gamma_j\gamma_{j+1}^{-q},
\gamma_{j+1}
\bigr),
\end{equation}
whose geometric triple is
\begin{equation}
\label{eq:termination-current-data}
\left(
l-m_1-m_2+m_j-(q-1)m_{j+1},
m_j-qm_{j+1},
m_{j+1}
\right).
\end{equation}
This pair is the first stopping state. It satisfies:
\begin{enumerate}[label=(\alph*)]
\item the positive short-overlap condition if the upper inequality in
(\ref{eq:l-inequality-main}) is strict;

\item the unequal-length equality condition if equality holds and
$m_j>(q+1)m_{j+1}$;

\item the equal-length condition if equality holds and
$m_j=(q+1)m_{j+1}$. This occurs only for $j=k-1$ and
$q=q_{k-1}-1$, equivalently $l=m_1+m_2-d$. The stopping triple is
$(m_k,m_k,m_k)$.
\end{enumerate}

\item If
\begin{equation}
\label{eq:termination-large-overlap}
l>m_1+m_2-d,
\end{equation}
the geometric triple at the stopping state $N$ is
\begin{equation}
\label{eq:termination-large-overlap-state}
\left(
l-m_1-m_2+2m_k,
m_k,
m_k
\right).
\end{equation}
\end{enumerate}
\end{thm}

\begin{proof}
If $l=0$, then the initial geometric triple already satisfies the
zero-overlap stopping condition, so no strict Nielsen subtraction is
performed. We therefore assume that $l>0$.

Recall that
\begin{equation}
\label{eq:termination-gcd-notation}
d\coloneqq\gcd(m_1,m_2).
\end{equation}
For $1\leq j\leq k-1$ and $-1\leq q\leq q_j-1$, define
\begin{equation}
\label{eq:termination-endpoint-definition}
B_{j,q}
\coloneqq
(m_1+m_2)-m_j+qm_{j+1}.
\end{equation}
The intervals
\begin{equation}
\label{eq:termination-interval-partition}
\bigl(B_{j,q-1},B_{j,q}\bigr],
\qquad
1\leq j\leq k-1,
\qquad
0\leq q\leq q_j-1,
\end{equation}
all have positive length $m_{j+1}$. They form a disjoint partition
of $\bigl(0,m_1+m_2-d\bigr]$. Indeed, the first left endpoint is
\begin{equation}
\label{eq:termination-first-endpoint}
B_{1,-1}=m_1+m_2-m_1-m_2=0,
\end{equation}
and the final right endpoint is
\begin{equation}
\label{eq:termination-last-endpoint}
B_{k-1,q_{k-1}-1}=m_1+m_2-d.
\end{equation}
Moreover, the right endpoint of each Euclidean block equals the
left endpoint of the next block:
\begin{equation}
\label{eq:termination-block-junction}
B_{j,q_j-1}=B_{j+1,-1},
\end{equation}
for $1\leq j\leq k-2$. This follows directly from
$m_j=q_jm_{j+1}+m_{j+2}$. Hence every
$0<l\leq m_1+m_2-d$ belongs to exactly one interval in
(\ref{eq:termination-interval-partition}), proving the existence
and uniqueness of $j$ and $q$ in (\ref{eq:l-inequality-main}).

Fix these unique indices $j$ and $q$, so that
\begin{equation}
\label{eq:l-inequality-jq}
\begin{split}
(m_1+m_2)-m_j+(q-1)m_{j+1}
&<l\\
&\leq
(m_1+m_2)-m_j+qm_{j+1}.
\end{split}
\end{equation}
By the formal Nielsen evolution established in Lemma~\ref{lem:euc},
the Nielsen pair associated with the indices $j$ and $q$ is
\begin{equation}
\label{eq:termination-proof-current-pair}
\bigl(
\gamma_j\gamma_{j+1}^{-q},
\gamma_{j+1}
\bigr).
\end{equation}

Let $n$ denote the uncompressed state index corresponding to
$(j,q)$. The formal translation-length pair at state $n$ and the
pair obtained after one further formal subtraction, before any
interchange, are
\begin{equation}
\label{eq:two-consecutive-lengths}
\begin{split}
(m_1^{(n)},m_2^{(n)})
&=
(m_j-qm_{j+1},m_{j+1}),\\
\bigl(
m_1^{(n)}-m_2^{(n)},
m_2^{(n)}
\bigr)
&=
(m_j-(q+1)m_{j+1},m_{j+1}).
\end{split}
\end{equation}
Define the corresponding formal overlap parameter by the invariant
relation (\ref{eq:invariant}):
\begin{equation}
\label{eq:l-n-expression}
\begin{split}
l^{(n)}
&=
l-m_1-m_2+m_1^{(n)}+m_2^{(n)}.
\end{split}
\end{equation}
Substituting (\ref{eq:two-consecutive-lengths}), the inequalities
in (\ref{eq:l-inequality-jq}) are equivalent to
\begin{equation}
\label{eq:termination-current-overlap-range}
0<l^{(n)}\leq m_{j+1}.
\end{equation}

We first identify the only normalized formal state at which the
two translation lengths can be equal. Such equality occurs precisely
when
\begin{equation}
\label{eq:termination-equal-length-relation}
m_j=(q+1)m_{j+1}.
\end{equation}
Combining this relation with (\ref{eq:Euc_algo}) gives
\begin{equation}
\label{eq:termination-euclidean-comparison}
(q+1)m_{j+1}=q_jm_{j+1}+m_{j+2},
\end{equation}
or equivalently
\begin{equation}
\label{eq:termination-euclidean-comparison-rearranged}
(q+1-q_j)m_{j+1}=m_{j+2}.
\end{equation}
Since $q\leq q_j-1$, the left-hand side is nonpositive, while
$m_{j+2}\geq0$. Both sides must therefore vanish. Hence
\begin{equation}
\label{eq:termination-final-endpoint-identification}
q=q_j-1,
\qquad
m_{j+2}=0,
\qquad
j=k-1.
\end{equation}
Conversely, these indices give equal translation lengths. Thus
every normalized formal state preceding
$(k-1,q_{k-1}-1)$ has strictly unequal translation lengths.

We now verify that the reduction reaches the state indexed by
$(j,q)$ before terminating. The strict predecessor inequalities
can be seen explicitly within and between Euclidean blocks.

If $q>0$, the preceding formal state lies in the same Euclidean
block and satisfies
\begin{equation}
\label{eq:termination-previous-same-block}
\begin{aligned}
m_1^{(n-1)}
&=
m_j-(q-1)m_{j+1}>m_{j+1}=m_2^{(n-1)},\\
l^{(n-1)}
&=
l^{(n)}+m_{j+1}>m_{j+1}.
\end{aligned}
\end{equation}
The first strict inequality follows from $q\leq q_j-1$ and
(\ref{eq:Euc_algo}), and the second follows from
(\ref{eq:termination-current-overlap-range}).

If $j>1$, consider the transition from the preceding Euclidean
block into the $j$-th block. Immediately before the final
subtraction in the preceding block, the translation lengths are
\begin{equation}
\label{eq:termination-previous-block-lengths}
(m_j+m_{j+1},m_j).
\end{equation}
The formal overlap parameter at that state is
\begin{equation}
\label{eq:termination-previous-block-overlap}
l-m_1-m_2+2m_j+m_{j+1}>m_j,
\end{equation}
because (\ref{eq:l-inequality-jq}) implies
$l>B_{j,-1}=m_1+m_2-m_j-m_{j+1}$.
Thus both strict reduction inequalities hold at this block
junction as well.

More generally, let $n'<n$ be any preceding normalized formal
state, with indices $(j',q')$. By the ordering of the intervals
in (\ref{eq:termination-interval-partition}),
\[
B_{j',q'}\leq B_{j,q-1}<l.
\]
The invariant and the formal translation lengths give
\[
\begin{split}
l^{(n')}-m_2^{(n')}
&=
l-m_1-m_2+m_{j'}-q'm_{j'+1}\\
&=
l-B_{j',q'}>0.
\end{split}
\]
Furthermore, the equal-length characterization
(\ref{eq:termination-final-endpoint-identification}) shows that
\[
m_1^{(n')}>m_2^{(n')},
\]
since $n'$ precedes the selected state and hence cannot be the
final equal-length state. Therefore every formal predecessor
satisfies both strict reduction conditions.

Starting with the initial geometric pair, induction on the
normalized-state index, using Proposition~\ref{prop:reduce_ii}
at each strict stage, shows that every predecessor is realized
geometrically and that the reduction reaches the state indexed
by $(j,q)$. Its formal data are consequently the actual
geometric triple at that stage.

If the right inequality in (\ref{eq:l-inequality-jq}) is strict,
then (\ref{eq:termination-current-overlap-range}) becomes
\begin{equation}
\label{eq:termination-strict-current-overlap}
0<l^{(n)}<m_{j+1}.
\end{equation}
The current geometric triple therefore satisfies the positive
short-overlap stopping condition, proving (a).

Suppose instead that equality holds on the right in
(\ref{eq:l-inequality-jq}). Then
\begin{equation}
\label{eq:termination-endpoint-current-overlap}
l^{(n)}=m_{j+1}.
\end{equation}
If
\begin{equation}
\label{eq:termination-endpoint-current-lengths}
m_j-qm_{j+1}>m_{j+1},
\end{equation}
the current geometric triple has overlap equal to its shorter
translation length and has unequal translation lengths.
It therefore satisfies the unequal-length equality stopping
condition, proving (b).

In the remaining case the two translation lengths are equal.
By (\ref{eq:termination-final-endpoint-identification}), this
occurs only when $j=k-1$ and $q=q_{k-1}-1$. The corresponding
right endpoint is
\begin{equation}
\label{eq:termination-proof-final-threshold}
l=m_1+m_2-d.
\end{equation}
The current translation lengths are both $d$, and
(\ref{eq:termination-endpoint-current-overlap}) gives overlap
$d$. Thus the stopping triple is $(d,d,d)$, proving (c).

Finally, suppose that
\begin{equation}
\label{eq:termination-proof-large-overlap}
l>m_1+m_2-d.
\end{equation}
Consider first the threshold value $l=m_1+m_2-d$.
The preceding argument shows that every normalized formal state
before the final equal-length state satisfies both strict
reduction conditions and is realized geometrically.

For the larger value of $l$ in
(\ref{eq:termination-proof-large-overlap}), the formal
translation-length sequence is unchanged. By the invariant
(\ref{eq:invariant}), the formal overlap parameter at each state
increases by the positive amount
$l-(m_1+m_2-d)$. Hence all strict reduction inequalities at
the preceding states remain valid. Induction using
Proposition~\ref{prop:reduce_ii} therefore shows that the
reduction again reaches the final equal-length state.

With this stopping state indexed by $N$, its translation lengths
are
\begin{equation}
\label{eq:pre-final-lengths}
m_1^{(N)}=m_2^{(N)}=d.
\end{equation}
By Lemma~\ref{lem:euc} and the invariant
(\ref{eq:invariant}), its overlap is
\begin{equation}
\label{eq:pre-final-l}
\begin{split}
l^{(N)}
&=
l-m_1-m_2+2d\\
&>
d.
\end{split}
\end{equation}
Thus the reduction reaches the final equal-length stopping state,
with the geometric triple asserted in (3). This completes the
proof.
\end{proof}

We now return to Example~\ref{ex:running-example}. The local argument
produced the normalized geometric evolution in
(\ref{eq:running_example_reduction}) directly. We now recover its
first stopping state from the Euclidean data.

For the initial translation lengths $m_1=13$ and $m_2=5$, the
Euclidean quotients are $(2,1,1,2)$, with successive remainders
$(3,2,1,0)$. Hence $d=\gcd(13,5)=1$, and the final threshold is
$13+5-1=17$. For $j=3$ and $q=0$, the interval in
Theorem~\ref{thm:termination} is $13<l\leq15$. Since the running
example has $l=14$, the theorem identifies the unique first
stopping pair as $(\gamma_3,\gamma_4)$, with geometric triple
$(1,3,2)$, corresponding to the final state of \eqref{eq:running_example_reduction}.

The formal Nielsen recurrence gives
$\gamma_3=\gamma_1\gamma_2^{-2}$ and
$\gamma_4=\gamma_2\gamma_3^{-1}$. Thus the first stopping pair is
\[
\bigl(
\gamma_1\gamma_2^{-2},
\gamma_2(\gamma_1\gamma_2^{-2})^{-1}
\bigr).
\]
Its overlap has length $1$, strictly smaller than the shorter
translation length $2$, so its stopping condition is positive
short overlap. Interchanging the generators gives the ordered pair $(\gamma_4,\gamma_3)$ with geometric triple $(1,2,3)$, which is prescribed by Theorem~\ref{thm:terminal-geometry}.
\subsection{Geometry of the terminal Nielsen pair}
\label{se:geo_ter_Nielsen_pair}

We now apply the local reduction results of
Section~\ref{sec:schottky-criterion} to the stopping pairs
identified by Theorem~\ref{thm:termination}. The following theorem
determines the terminal Nielsen pair in each stopping case,
including the displacement types of its generators and the
configuration of their characteristic sets.

\begin{thm}[Geometry of the terminal Nielsen pair]
\label{thm:terminal-geometry}
Let $(\gamma_1,\gamma_2)$ be the normalized initial hyperbolic pair,
and let Theorem~\ref{thm:termination} supply its unique first stopping
state. Set $d=\gcd(m_1,m_2)$. In the short-overlap case, we may
interchange the first stopping pair and record the terminal pair in
the order used below. Then exactly one of the following occurs.

\begin{enumerate}[label=(\roman*)]

\item \emph{Short overlap.}
If $l=0$, the terminal pair is $(\gamma_1,\gamma_2)$, and its axes
are edge-disjoint.

Suppose that $l>0$, and let $j$ and $q$ be the indices supplied by
Theorem~\ref{thm:termination}, with strict upper inequality in
(\ref{eq:l-inequality-main}). The first stopping pair is
$(\gamma_j\gamma_{j+1}^{-q},\gamma_{j+1})$. After interchanging the
generators, the terminal pair is
\[
(\gamma_{\mathrm a},\gamma_{\mathrm b})
=
(\gamma_{j+1},\gamma_j\gamma_{j+1}^{-q}).
\]
Both generators are hyperbolic, with translation lengths
\[
m_{\gamma_{\mathrm a}}=m_{j+1},
\qquad
m_{\gamma_{\mathrm b}}=m_j-qm_{j+1}.
\]
Their axes overlap in a segment of length
\[
l-m_1-m_2+m_j-(q-1)m_{j+1}.
\]
This length is smaller than the shorter translation length, and
$v^+$ is an endpoint of the terminal axes intersection.

\item \emph{Unequal-length equality.}
Let $j$ and $q$ be the indices supplied by
Theorem~\ref{thm:termination}. Assume that equality holds in
(\ref{eq:l-inequality-main}) and that
$m_j>(q+1)m_{j+1}$. The preterminal pair is
\[
(h,\gamma_{\mathrm a})
=
(\gamma_j\gamma_{j+1}^{-q},\gamma_{j+1}),
\]
with geometric triple
\[
(m_{j+1},m_j-qm_{j+1},m_{j+1}).
\]
The terminal pair is
\[
(\gamma_{\mathrm a},\gamma_{\mathrm b})
=
(\gamma_{j+1},\gamma_j\gamma_{j+1}^{-(q+1)}).
\]

Set
\[
l_0=\ell(A_h\cap\gamma_{j+1}.A_h)
\in\mathbb N_{\geq 0}\cup\{+\infty\}.
\]
If
\[
l_0\geq\frac{m_j-(q+1)m_{j+1}}{2},
\]
then $\gamma_{\mathrm b}$ is a nontrivial elliptic automorphism and
\[
d(A_{\gamma_{\mathrm a}},T_{\gamma_{\mathrm b}})
=
\frac{m_j-(q+1)m_{j+1}}{2}.
\]

If
\[
l_0<\frac{m_j-(q+1)m_{j+1}}{2},
\]
then $\gamma_{\mathrm b}$ is hyperbolic, with translation length
\[
m_{\gamma_{\mathrm b}}
=
m_j-(q+1)m_{j+1}-2l_0.
\]
The terminal axes are edge-disjoint and satisfy
\[
d(A_{\gamma_{\mathrm a}},A_{\gamma_{\mathrm b}})=l_0.
\]

\item \emph{Final gcd endpoint.}
Assume that $l=m_1+m_2-d$. The preterminal pair is
$(\gamma_{k+1}\gamma_k,\gamma_k)$ and has geometric triple
$(d,d,d)$. The terminal pair is
\[
(\gamma_{\mathrm a},\gamma_{\mathrm b})
=
(\gamma_k,\gamma_{k+1})
=
(\gamma_k,\gamma_{k-1}\gamma_k^{-q_{k-1}}).
\]
Here $\gamma_{\mathrm a}$ is hyperbolic and
$\gamma_{\mathrm b}$ is elliptic. Moreover,
\[
A_{\gamma_{\mathrm a}}\cap T_{\gamma_{\mathrm b}}
=
\{v^+\},
\qquad
d(A_{\gamma_{\mathrm a}},T_{\gamma_{\mathrm b}})=0.
\]

\item \emph{Strict large overlap.}
Assume that $l>m_1+m_2-d$. The preterminal pair is
$(\gamma_{k+1}\gamma_k,\gamma_k)$ and has geometric triple
\[
(l-m_1-m_2+2d,d,d).
\]
The terminal pair is
\[
(\gamma_{\mathrm a},\gamma_{\mathrm b})
=
(\gamma_k,\gamma_{k+1}).
\]
Here $\gamma_{\mathrm a}$ is hyperbolic and
$\gamma_{\mathrm b}$ is elliptic. The intersection
$A_{\gamma_{\mathrm a}}\cap T_{\gamma_{\mathrm b}}$
is a nondegenerate segment of length $l-m_1-m_2+d$,
and $v^+$ is one of its endpoints.

\end{enumerate}

In every case, the terminal pair is Nielsen-equivalent to the
initial pair and therefore generates $\Gamma$.
\end{thm}

\begin{proof}
Throughout the proof, we retain the normalization fixed in
Subsection~\ref{se:f_euc_evo_stop}. The normalized initial pair
has $m_1\geq m_2>0$, and, when $l>0$, the common orientation and
the distinguished endpoint $v^+$ are fixed before the Euclidean
and Nielsen sequences are defined. By
Proposition~\ref{prop:reduce_ii}, every realized strict subtraction
produces the prescribed ordered Nielsen pair while preserving the
common orientation and the distinguished endpoint $v^+$. Each
block-end interchange restores the length order and preserves
the generated subgroup, the overlap parameter, and the common
orientation. Thus no further generator inversion is required
during the strict evolution.

Let
\begin{equation}
\label{eq:theorem-proof-pair-sequence}
\bigl(
\gamma_1^{(n)},\gamma_2^{(n)}
\bigr)
\end{equation}
denote the normalized ordered pair at the $n$-th realized geometric
state, including the first stopping state. Let
\begin{equation}
\label{eq:theorem-proof-data-sequence}
\bigl(
l^{(n)},m_1^{(n)},m_2^{(n)}
\bigr)
\end{equation}
denote its actual geometric triple, and write $N$ for the index
of the first stopping state. By Lemma~\ref{lem:euc},
Theorem~\ref{thm:termination}, and
Proposition~\ref{prop:reduce_ii}, the formal Nielsen evolution
in (\ref{eq:gamma-sequence}) is geometrically realized through
state $N$ by the corresponding ordered hyperbolic pairs in
(\ref{eq:theorem-proof-pair-sequence}). Every realized state
with $n<N$ satisfies the strict reduction conditions, and the
required block-end interchanges are performed.

\medskip
\noindent\textbf{Initial short overlap.}
Suppose first that
\begin{equation}
\label{eq:theorem-proof-initial-short-overlap}
l<m_2.
\end{equation}
The initial pair is already the first stopping pair. If $l=0$,
its axes are edge-disjoint under the standing convention, and
the terminal pair is the original pair.

Suppose that $0<l<m_2$. Then the upper inequality in
(\ref{eq:l-inequality-jq}) is strict with $j=1$ and $q=0$.
After interchanging the generators, the terminal pair is
\begin{equation}
\label{eq:theorem-proof-initial-terminal-pair}
(\gamma_{\mathrm a},\gamma_{\mathrm b})
=
(\gamma_{j+1},\gamma_j\gamma_{j+1}^{-q})
=
(\gamma_2,\gamma_1).
\end{equation}
Both generators are hyperbolic, their axes overlap in a segment
of length $l$, and $v^+$ is an endpoint of that segment by the
initial normalization.

\medskip
\noindent\textbf{Positive short overlap after strict reduction.}
Assume that the first stopping state is a positive short-overlap
state and that at least one strict subtraction occurs. By
Theorem~\ref{thm:termination}, there are unique $j$ and $q$,
with $(j,q)\neq(1,0)$, such that
\begin{equation}
\label{eq:l-inequality-caseII}
\begin{split}
(m_1+m_2)-m_j+(q-1)m_{j+1}
&<l\\
&<
(m_1+m_2)-m_j+qm_{j+1}.
\end{split}
\end{equation}
The first stopping pair, before interchanging the generators, is
\begin{equation}
\label{eq:theorem-proof-caseII-pair}
(\gamma_{\mathrm b},\gamma_{\mathrm a})
=
\bigl(
\gamma_j\gamma_{j+1}^{-q},
\gamma_{j+1}
\bigr).
\end{equation}
After interchange, this gives the terminal pair
$(\gamma_{\mathrm a},\gamma_{\mathrm b})$ used in the theorem.
The actual geometric triple at the first stopping state, before
interchange, is
\begin{equation}
\label{eq:triple-at-termination-caseII}
\begin{split}
\bigl(
l^{(N)},m_1^{(N)},m_2^{(N)}
\bigr)
=
\bigl(
&l-m_1-m_2+m_j-(q-1)m_{j+1},\\
&m_j-qm_{j+1},
m_{j+1}
\bigr).
\end{split}
\end{equation}
Hence
\begin{equation}
\label{eq:theorem-proof-caseII-terminal-lengths}
m_{\gamma_{\mathrm a}}=m_{j+1},
\qquad
m_{\gamma_{\mathrm b}}=m_j-qm_{j+1}.
\end{equation}
The inequalities in (\ref{eq:l-inequality-caseII}) give
\begin{equation}
\label{eq:theorem-proof-caseII-overlap-range}
0<l^{(N)}<m_{j+1}\leq m_j-qm_{j+1}.
\end{equation}
Thus both terminal generators are hyperbolic, and their exact
axis overlap is
\begin{equation}
\label{eq:theorem-proof-caseII-overlap}
\ell\bigl(
A_{\gamma_{\mathrm a}}\cap A_{\gamma_{\mathrm b}}
\bigr)
=
l-m_1-m_2+m_j-(q-1)m_{j+1}.
\end{equation}
At every strict reduction, Proposition~\ref{prop:reduce_ii}
preserves $v^+$ as an endpoint of the new axes intersection.
Each block-end interchange leaves the two axes and their
intersection unchanged. Hence $v^+$ is an endpoint of
$A_{\gamma_{\mathrm a}}\cap A_{\gamma_{\mathrm b}}$.

\medskip
\noindent\textbf{Nonfinal unequal-length equality.}
Suppose that equality holds in the upper bound of
(\ref{eq:l-inequality-jq}) at the first stopping state, namely,
\begin{equation}
\label{eq:l-caseIII}
l=(m_1+m_2)-m_j+qm_{j+1},
\end{equation}
and that
\begin{equation}
\label{eq:theorem-proof-caseIII-excluded-endpoint}
(j,q)\neq(k-1,q_{k-1}-1).
\end{equation}
By Theorem~\ref{thm:termination}, the two translation lengths
at this state are unequal, so
\begin{equation}
\label{eq:theorem-proof-caseIII-unequal-lengths}
m_j-qm_{j+1}>m_{j+1}.
\end{equation}
The preterminal hyperbolic pair is
\begin{equation}
\label{eq:theorem-proof-caseIII-preterminal-pair}
(h,\gamma_{\mathrm a})
=
\bigl(
\gamma_j\gamma_{j+1}^{-q},
\gamma_{j+1}
\bigr),
\end{equation}
with actual geometric triple
\[
\bigl(
m_{j+1},m_j-qm_{j+1},m_{j+1}
\bigr).
\]
The terminal Nielsen subtraction gives, before interchange,
\begin{equation}
\label{eq:theorem-proof-caseIII-terminal-pair}
(\gamma_{\mathrm b},\gamma_{\mathrm a})
=
(h\gamma_{\mathrm a}^{-1},\gamma_{\mathrm a})
=
\bigl(
\gamma_j\gamma_{j+1}^{-(q+1)},
\gamma_{j+1}
\bigr).
\end{equation}
After interchange, this is the terminal pair
$(\gamma_{\mathrm a},\gamma_{\mathrm b})$ used in the theorem.

The actual geometry after this terminal subtraction need not
agree with the formal post-subtraction data. To determine it,
define
\begin{equation}
\label{eq:theorem-proof-caseIII-l-zero}
l_0
\coloneqq
\ell\bigl(A_h\cap\gamma_{j+1}.A_h\bigr)
\in\mathbb N_{\geq 0}\cup\{+\infty\},
\end{equation}
where $l_0=+\infty$ if the intersection contains an infinite ray.

By Proposition~\ref{prop:reduce_iii}, if
\begin{equation}
\label{eq:theorem-proof-caseIII-elliptic-condition}
l_0
\geq
\frac{m_j-(q+1)m_{j+1}}{2},
\end{equation}
then
$\gamma_{\mathrm b}=\gamma_j\gamma_{j+1}^{-(q+1)}$
is elliptic and
\begin{equation}
\label{eq:theorem-proof-caseIII-axis-tree-distance}
d\bigl(
A_{\gamma_{\mathrm a}},
T_{\gamma_{\mathrm b}}
\bigr)
=
\frac{m_j-(q+1)m_{j+1}}{2}.
\end{equation}
This distance is positive by
(\ref{eq:theorem-proof-caseIII-unequal-lengths}), so
$\gamma_{\mathrm b}$ is nontrivial.

If instead
\begin{equation}
\label{eq:theorem-proof-caseIII-hyperbolic-condition}
l_0
<
\frac{m_j-(q+1)m_{j+1}}{2},
\end{equation}
then $\gamma_{\mathrm b}$ is hyperbolic, with translation length
\begin{equation}
\label{eq:theorem-proof-caseIII-hyperbolic-length}
m_{\gamma_{\mathrm b}}
=
m_j-(q+1)m_{j+1}-2l_0.
\end{equation}
The terminal axes are edge-disjoint and satisfy
\begin{equation}
\label{eq:theorem-proof-caseIII-axis-distance}
d\bigl(
A_{\gamma_{\mathrm a}},
A_{\gamma_{\mathrm b}}
\bigr)=l_0.
\end{equation}

\medskip
\noindent\textbf{Equal-length terminal reduction.}
Set $d\coloneqq\gcd(m_1,m_2)=m_k$. By
Theorem~\ref{thm:termination}, equal-length stopping occurs
precisely when
\begin{equation}
\label{eq:theorem-proof-caseIV-range}
l\geq m_1+m_2-d.
\end{equation}
The actual preterminal pair is
$(\gamma_{k+1}\gamma_k,\gamma_k)$, with geometric triple
\[
(l-m_1-m_2+2d,d,d).
\]
The final subtraction and interchange give
\begin{equation}
\label{eq:theorem-proof-caseIV-terminal-pair}
(\gamma_{\mathrm a},\gamma_{\mathrm b})
=
(\gamma_k,\gamma_{k+1}),
\end{equation}
where $\gamma_{\mathrm a}$ is hyperbolic and
$\gamma_{\mathrm b}$ is elliptic. The formal post-subtraction
data are
\[
\bigl(l-m_1-m_2+d,d,0\bigr).
\]
These are formal data, not an actual geometric triple after
the terminal subtraction.

Theorem~\ref{thm:termination} and the invariant
(\ref{eq:invariant}) give
\begin{equation}
\label{eq:termination-proof-fixed-length}
l^{(N)}-d
=
l-m_1-m_2+d
=
l-m_1-m_2+m_k.
\end{equation}
Consequently, Proposition~\ref{prop:reduce_iv} gives
\begin{equation}
\label{eq:theorem-proof-caseIV-axis-tree-overlap}
\ell\bigl(
A_{\gamma_{\mathrm a}}\cap T_{\gamma_{\mathrm b}}
\bigr)
=
l-m_1-m_2+d.
\end{equation}

At every strict reduction, Proposition~\ref{prop:reduce_ii}
preserves $v^+$ as the distinguished endpoint. Each block-end
interchange preserves the unordered characteristic-set
configuration. By Proposition~\ref{prop:reduce_iv}, the
intersection
$A_{\gamma_{\mathrm a}}\cap T_{\gamma_{\mathrm b}}$
is a segment with endpoint $v^+$ and length
$l-m_1-m_2+d\geq0$.

If $l>m_1+m_2-d$, this segment is nondegenerate, giving
case (iv). If $l=m_1+m_2-d$, the intersection is the singleton
$\{v^+\}$, giving case (iii).

In every alternative, the terminal pair is obtained from the
normalized initial pair by Nielsen moves. Hence it is
Nielsen-equivalent to the initial pair and generates $\Gamma$.
This proves all the terminal-geometry alternatives.

 \renewcommand{\qedsymbol}{}
\namedqedline{Theorem~\ref{thm:terminal-geometry}}
\end{proof}

The terminal Nielsen pairs, their displacement types and translation
data, and the geometry of their characteristic sets are summarized in
Table~\ref{tab:main-cases}. Here $C_g=A_g$ if $g$ is hyperbolic and
$C_g=T_g$ if $g$ is elliptic; also, $d=\gcd(m_1,m_2)$, and $l_0$ in
case~(ii) is defined in Theorem~\ref{thm:terminal-geometry}.

\begin{table}[H]
\centering
\small
\setlength{\tabcolsep}{3pt}
\renewcommand{\arraystretch}{1.25}
\begin{tabular}{|c||c|c|c|}
\hline
Case
&
Terminal pair
&
Type and translation data
&
Characteristic-set geometry
\\
\hline\hline
(i), $l=0$
&
$(\gamma_1,\gamma_2)$
&
\begin{tabular}[c]{@{}c@{}}
$m_{\gamma_1}=m_1$\\
$m_{\gamma_2}=m_2$
\end{tabular}
&
\begin{tabular}[c]{@{}c@{}}
$C_{\gamma_1}$ and $C_{\gamma_2}$ are edge-disjoint
\end{tabular}
\\
\hline
(i), $l>0$
&
\begin{tabular}[c]{@{}c@{}}
$(\gamma_{j+1},$\\
$\gamma_j\gamma_{j+1}^{-q})$
\end{tabular}
&
\begin{tabular}[c]{@{}c@{}}
$m_{\gamma_{\mathrm a}}=m_{j+1}$\\
$m_{\gamma_{\mathrm b}}=m_j-qm_{j+1}$
\end{tabular}
&
\begin{tabular}[c]{@{}c@{}}
$\ell(C_{\gamma_{\mathrm a}}
\cap C_{\gamma_{\mathrm b}})$\\
$=l-m_1-m_2+m_j-(q-1)m_{j+1}$;\\
$0<\ell(C_{\gamma_{\mathrm a}}
\cap C_{\gamma_{\mathrm b}})<m_{\gamma_{\mathrm a}}
\leq m_{\gamma_{\mathrm b}}$;\\
$v^+$ is an endpoint
\end{tabular}
\\
\hline
(ii), elliptic
&
\begin{tabular}[c]{@{}c@{}}
$(\gamma_{j+1},$\\
$\gamma_j\gamma_{j+1}^{-(q+1)})$
\end{tabular}
&
\begin{tabular}[c]{@{}c@{}}
$l_0\geq
\bigl(m_j-(q+1)m_{j+1}\bigr)/2$\\
$m_{\gamma_{\mathrm a}}=m_{j+1}$\\
$\gamma_{\mathrm b}$ is elliptic
\end{tabular}
&
\begin{tabular}[c]{@{}c@{}}
$C_{\gamma_{\mathrm a}}$ and
$C_{\gamma_{\mathrm b}}$ are disjoint;\\
$d(C_{\gamma_{\mathrm a}},
C_{\gamma_{\mathrm b}})$\\
$=(m_j-(q+1)m_{j+1})/2>0$
\end{tabular}
\\
\hline
(ii), hyperbolic
&
\begin{tabular}[c]{@{}c@{}}
$(\gamma_{j+1},$\\
$\gamma_j\gamma_{j+1}^{-(q+1)})$
\end{tabular}
&
\begin{tabular}[c]{@{}c@{}}
$l_0<
\bigl(m_j-(q+1)m_{j+1}\bigr)/2$\\
$m_{\gamma_{\mathrm a}}=m_{j+1}$\\
$m_{\gamma_{\mathrm b}}=m_j-(q+1)m_{j+1}-2l_0$
\end{tabular}
&
\begin{tabular}[c]{@{}c@{}}
$C_{\gamma_{\mathrm a}}$ and
$C_{\gamma_{\mathrm b}}$ are edge-disjoint;\\
$d(C_{\gamma_{\mathrm a}},
C_{\gamma_{\mathrm b}})=l_0$
\end{tabular}
\\
\hline
(iii)
&
$(\gamma_k,\gamma_{k+1})$
&
\begin{tabular}[c]{@{}c@{}}
$m_{\gamma_{\mathrm a}}=d$\\
$\gamma_{\mathrm b}$ is elliptic
\end{tabular}
&
\begin{tabular}[c]{@{}c@{}}
$C_{\gamma_{\mathrm a}}
\cap C_{\gamma_{\mathrm b}}=\{v^+\}$;\\
$d(C_{\gamma_{\mathrm a}},
C_{\gamma_{\mathrm b}})=0$
\end{tabular}
\\
\hline
(iv)
&
$(\gamma_k,\gamma_{k+1})$
&
\begin{tabular}[c]{@{}c@{}}
$m_{\gamma_{\mathrm a}}=d$\\
$\gamma_{\mathrm b}$ is elliptic
\end{tabular}
&
\begin{tabular}[c]{@{}c@{}}
$C_{\gamma_{\mathrm a}}
\cap C_{\gamma_{\mathrm b}}$\\
is a nondegenerate segment;\\
$\ell(C_{\gamma_{\mathrm a}}
\cap C_{\gamma_{\mathrm b}})$\\
$=l-m_1-m_2+d>0$;\\
$v^+$ is an endpoint
\end{tabular}
\\
\hline
\end{tabular}
\vspace{.1in}
\caption{Terminal Nielsen pairs and their characteristic-set geometry.}
\label{tab:main-cases}
\end{table}

\medskip

\subsection{The Schottky--torsion dichotomy}
\label{se:schottky-torsion-dichotomy}
\begin{thm}[Main theorem]
\label{thm:main}
Let $X$ be a locally finite tree, and let
$\gamma_1,\gamma_2\in\aut(X)$ be hyperbolic automorphisms. Assume that
the action is without inversions, that
$A_{\gamma_1}\cap A_{\gamma_2}$ has finite edge intersection, and that
$\Gamma=\langle\gamma_1,\gamma_2\rangle$ is discrete. Set
$m_i=m_{\gamma_i}$ for $i=1,2$, and set
$l=\ell(A_{\gamma_1}\cap A_{\gamma_2})$, with the standing convention
$l=0$ when the axes are edge-disjoint.

Assume that $(\gamma_1,\gamma_2)$ is a normalized pair as fixed in
Subsection~\ref{se:f_euc_evo_stop}; that is, $m_1\geq m_2>0$, and
the two generators induce the same orientation on their common
axis segment if $l>0$.

Let $(\gamma_{\mathrm a},\gamma_{\mathrm b})$ be the terminal Nielsen
pair supplied by Theorem~\ref{thm:terminal-geometry} from the unique
first stopping state supplied by Theorem~\ref{thm:termination}. Then
exactly one of the following occurs.

\begin{enumerate}[label=(\arabic*)]
\item Both terminal generators are hyperbolic. Then
$(\gamma_{\mathrm a},\gamma_{\mathrm b})$ is a Schottky pair, and
$\Gamma$ is a rank-two Schottky group.

\item Exactly one terminal generator is nontrivial elliptic. Then it
has finite order, and $\Gamma$ is not free of rank two.
\end{enumerate}

The numerical conditions and the exact terminal geometry in these
alternatives are those of Theorem~\ref{thm:termination} and
Theorem~\ref{thm:terminal-geometry}.
\end{thm}

\begin{proof}
Theorem~\ref{thm:termination} supplies the unique first stopping
state, and Theorem~\ref{thm:terminal-geometry} supplies the
corresponding terminal Nielsen pair and its exact characteristic-set
geometry, summarized in Table~\ref{tab:main-cases}.

It remains to establish the group-theoretic conclusion in each
terminal alternative.

In case (i) and the hyperbolic branch of case (ii), both terminal
generators are hyperbolic, and the length of their axis intersection
is strictly smaller than both translation lengths. After
interchanging the generators if necessary to meet the ordering
convention of Proposition~\ref{prop:schottky}, that proposition
shows that the terminal pair
$(\gamma_{\mathrm a},\gamma_{\mathrm b})$ is Schottky.
Since the terminal pair is Nielsen-equivalent to the initial pair,
it generates $\Gamma$. Hence $\Gamma$ is a rank-two Schottky group.

In the elliptic branch of case (ii), case (iii), and case (iv),
the terminal pair consists of a hyperbolic generator
$\gamma_{\mathrm a}$ and an elliptic generator
$\gamma_{\mathrm b}$. We first verify that
$\gamma_{\mathrm b}$ is nontrivial. In the elliptic branch of
case (ii), Theorem~\ref{thm:terminal-geometry} gives
\[
d\bigl(
A_{\gamma_{\mathrm a}},T_{\gamma_{\mathrm b}}
\bigr)>0.
\]
In case (iii), the same theorem gives
\[
A_{\gamma_{\mathrm a}}\cap T_{\gamma_{\mathrm b}}
=\{v^+\},
\]
while in case (iv), this intersection is a finite
nondegenerate segment. None of these configurations is possible
if $\gamma_{\mathrm b}=1$, since the identity fixes all of $X$
and hence
\[
A_{\gamma_{\mathrm a}}\cap T_1
=A_{\gamma_{\mathrm a}},
\]
which is a bi-infinite geodesic. Thus
$\gamma_{\mathrm b}$ is nontrivial in every elliptic alternative.

Since $\gamma_{\mathrm b}$ is elliptic and the action is without
inversions, it fixes a vertex $v$. The local finiteness of $X$
and discreteness of $\Gamma$ imply that the vertex stabilizer
$\Gamma_v$ is finite. Therefore $\gamma_{\mathrm b}$ has finite
order. As $\gamma_{\mathrm b}\neq1$, the group $\Gamma$ contains
nontrivial torsion and hence is not free of rank two.

Exhaustivity follows from Theorem~\ref{thm:termination}, and mutual
exclusivity follows from the terminal alternatives in
Theorem~\ref{thm:terminal-geometry}. This proves the classification.
\renewcommand{\qedsymbol}{}
\namedqedline{Theorem~\ref{thm:main}}
\end{proof}

We now conclude Example~\ref{ex:running-example}. The first stopping
pair found above is $(\gamma_3,\gamma_4)$, with translation lengths
$(3,2)$ and axis overlap $1$. Since this is the short-overlap case,
Theorem~\ref{thm:terminal-geometry} records the terminal pair in the
order $(\gamma_4,\gamma_3)$, with translation lengths $(2,3)$ and
the same axis overlap $1$. Here
$\gamma_3=\gamma_1\gamma_2^{-2}$ and
$\gamma_4=\gamma_2\gamma_3^{-1}$, so the terminal Nielsen generators
are explicit words in the original generators.

The terminal pair is Schottky by Proposition~\ref{prop:schottky}.
Since it is Nielsen-equivalent to the initial pair, it generates
$\Gamma$. Theorem~\ref{thm:main} therefore places the running example
in the Schottky alternative and shows that $\Gamma$ is a rank-two
Schottky group.

\begin{rmk}
\label{rmk:shimizu-criterion}
A standard ping-pong argument gives the following sufficient
criterion. If $g,h\in\aut(X)$ are hyperbolic and
\begin{equation}
\label{eq:serre-shimizu}
\ell(A_g\cap A_h)<\min\{m_g,m_h\},
\end{equation}
then $\langle g,h\rangle$ is free of rank two; see
\cite[Lemma~2.6]{CuMo}.

Theorem~\ref{thm:main} refines this criterion within the present
setting. The Nielsen evolution in (\ref{eq:group_ele_seq}) is
geometrically realized through the first stopping state. Together
with the additional terminal Nielsen move in the equality cases,
this determines whether the terminal pair satisfies the
short-overlap criterion or contains a nontrivial elliptic generator.

In particular, if the initial overlap is at or above the final
Euclidean threshold, that is,
\begin{equation}
\label{eq:main-threshold}
\ell(A_{\gamma_1}\cap A_{\gamma_2})
\geq m_1+m_2-\gcd(m_1,m_2),
\end{equation}
then the terminal pair contains a nontrivial elliptic generator.

Below this threshold, the terminal pair is Schottky except
possibly at the nonfinal Euclidean endpoints. At each such
endpoint, the auxiliary overlap $l_0$ determines whether the
terminal pair is Schottky or contains a nontrivial elliptic
generator, as specified in Theorem~\ref{thm:terminal-geometry}.

The classification may be summarized as follows, where
$d=\gcd(m_1,m_2)$:

\begin{center}
\small
\renewcommand{\arraystretch}{1.2}
\begin{tabular}{@{}p{0.49\linewidth}p{0.43\linewidth}@{}}
\hline
Initial overlap regime & Terminal outcome\\
\hline
$0\leq l<m_1+m_2-d$, away from nonfinal Euclidean endpoints
&
Schottky pair
\\
$l<m_1+m_2-d$, at a nonfinal Euclidean endpoint
&
Schottky pair or nontrivial elliptic generator,
according to $l_0$
\\
$l\geq m_1+m_2-d$
&
Nontrivial elliptic generator
\\
\hline
\end{tabular}
\end{center}

\end{rmk}

\end{document}